\documentclass[a4paper]{amsart} 

\usepackage[utf8]{inputenc}
\usepackage[T1]{fontenc}
\usepackage{amssymb}
\usepackage{colonequals}
\usepackage[pdfencoding=auto]{hyperref}
\hypersetup{hypertexnames = false, bookmarksdepth = 2, bookmarksopen = true, colorlinks, linkcolor = black, citecolor = black, urlcolor = black, pdfstartview={XYZ null null 1}}
\usepackage[capitalise]{cleveref}

\allowdisplaybreaks

\theoremstyle{plain}
\newtheorem{theor}{Theorem}[section]
\newtheorem{lemma}[theor]{Lemma}
\newtheorem{corol}[theor]{Corollary}
\newtheorem{propo}[theor]{Proposition}
\newtheorem{remar}[theor]{Remark}
\theoremstyle{definition}
\newtheorem{defin}[theor]{Definition}
\newtheorem{examp}[theor]{Example}

\numberwithin{equation}{section}

\newcommand{\ca}[1]{\mathcal{#1}}
\newcommand{\bb}[1]{\mathbb{#1}}
\newcommand{\fr}[1]{\mathfrak{#1}}

\newcommand{\cC}{\mathcal{C}}

\newcommand{\bC}{\mathbb{C}}
\newcommand{\bZ}{\mathbb{Z}}

\newcommand{\fh}{\mathfrak{h}}

\newcommand{\rO}{\mathrm{O}}
\newcommand{\Pin}{\mathrm{Pin}}
\newcommand{\Spin}{\mathrm{Spin}}

\newcommand{\p}[1]{\rlap{\,#1}}

\newcommand{\g}[1]{\gamma_{#1}}

\newcommand{\gO}[2]{\g{#1} O_{#2} - \g{#2}O_{#1}}
\newcommand{\cA}[1]{\mathcal{A}( \gamma_{#1})}
\newcommand{\ib}[3][A]{\smash{\overset{#2}{#3\vphantom{#1}}}}
\newcommand{\ascom}[3][\,\,\,]{\{#2\ib[.]{#1\,}{,}#3\}}
\newcommand{\scom}[3][\,\,\,]{[#2\ib[.]{#1\,}{,}#3]}
\newcommand{\GG}[2]{\ascom[-]{O_{#1}}{O_{#2}}} 
\newcommand{\oO}{\tilde\sigma} 

\DeclareMathOperator{\Cent}{Cent}
\DeclareMathOperator{\Id}{Id}

\DeclareMathOperator{\End}{End}

\DeclareMathOperator{\Span}{span}
\DeclareMathOperator{\sgn}{sign}

\begin{document}
	
\title[Supercentralizers for deformations of the Pin osp dual pair ]{Supercentralizers for deformations of the Pin osp dual pair}
\author{Roy Oste}
\address{Department of Applied Mathematics, Computer Science and Statistics, Faculty of Sciences, Ghent University, Krijgslaan 281, Building S9, 9000 Gent, Belgium}
\email{Roy.Oste@UGent.be}
\thanks{
	The author was supported by a postdoctoral fellowship, fundamental research, of the Research Foundation -- Flanders (FWO), number 12Z9920N}
\begin{abstract}
In recent work, we examined the algebraic structure underlying a class of elements supercommuting with a realization of the Lie superalgebra $\fr{osp}(1|2)$ inside a generalization of the Weyl Clifford algebra. 
This  generalization contained in particular the deformation by means of Dunkl operators associated with a real reflection group, yielding a rational Cherednik algebra instead of the Weyl algebra. 
The aim of this work is to show that this is the full supercentralizer, give a (minimal) set of generators, and to describe the relation with the $(\Pin(d),\fr{osp}(2m+1|2n))$ Howe dual pair. 
\end{abstract}

\maketitle

%
%
%
%

\section{Introduction}


In recent work~\cite{DOV}, we started the investigation of the algebraic structure underlying a class of elements supercommuting with a realization of the Lie superalgebra $\fr{osp}(1|2)$ inside a generalization of the Weyl Clifford algebra. 
This generalization contained in particular the deformation by means of Dunkl operators associated with a real reflection group, yielding a rational Cherednik algebra instead of the Weyl algebra. 
The aim of this work is to show that this is the full supercentralizer inside the tensor product of rational Cherednik  $H_\kappa$ and a Clifford algebra, give a (minimal) set of generators, and to describe the relation with the $(\Pin(d),\fr{osp}(2m+1|2n))$ Howe dual pair.

%





An explicit realisation of $H_\kappa$ is given by means of Dunkl operators (for the elements of $V$ )
and coordinate variables (the elements of $V^*$), which gives a natural (faithful) action on  the polynomial space $S(V^*)$. 

We will work over $\mathbb{C}$, the field of complex numbers with  $i^2=-1$. 

Throughout, $[\cdot,\cdot]$ will denote the skew-supersymmetric operation on a Lie superalgebra or the supercommutator~\eqref{e:supercomm}.
The notation $\{\cdot,\cdot\}$ will denote the antisupercommutator~\eqref{e:ascom}. 
Moreover, a sign above the comma will sometimes, mostly in Section~\ref{s:rels}, be used to indicate the actual sign used in a(n anti)supercommutator. For instance, if $a$ and $b$ are odd then 
$\scom{a}{b} = ab + ba $, so we will write $\scom[+]{a}{b} $, while if $a$ or $b$ is even, we have $\scom[-]{a}{b} = ab - ba $.

Tensor products are assumed to be $\bZ_2$-graded, unless stated otherwise. The notation $ \odot $ will be used for the supersymmetric tensor product~\eqref{e:supersymtensor}.

Notations are not final. 
 $O_u = \oO_u$ for $u\in V^*$, 


%
%
%
%
%
%
%
%
%

\section{Lie superalgebras}\label{s:LSA}

\subsection{Preliminaries}

Denote  $\bar 0$ and $\bar 1$ the  elements of $\mathbb{Z}_2 = \mathbb{Z}/2\mathbb{Z}$, the residue class ring mod 2.  
The term superspace is used to refer to a $\mathbb{Z}_2$-graded vector space  $\ca V =\ca  V_{\bar 0 } \oplus\ca  V_{\bar 1}$, and superalgebra for a $\mathbb{Z}_2$-graded algebra. The parity or ($\mathbb{Z}_2$-)degree of a homogeneous element $a \in \ca  V_j$ is denoted by $|a| = j \in \bZ_2$. 
The parity reversing functor $\Pi$ sends a superspace $\ca V = \ca V_{\bar 0 } \oplus\ca  V_{\bar 1}$ to the superspace $\Pi(\ca V)$ with the opposite $\mathbb{Z}_2$-grading: 
$\Pi(\ca V_j) = \ca V_{j+\bar 1}$ for $j\in\mathbb{Z}_2$.

A Lie superalgebra $\mathfrak{g}= \mathfrak{g}_{\bar 0}  \oplus \mathfrak{g}_{\bar 1}$ is a superalgebra whose product $[\cdot,\cdot]$ satisfies, for homogeneous elements $x,y\in  \mathfrak{g}_{\bar 0}  \cup \mathfrak{g}_{\bar 1}$ and $z\in  \mathfrak{g}$,
\begin{align}
	[x,y] & = -(-1)^{|x||y|}[y,x] &&\text{(super skew-symmetry)} \\*
	[x,[y,z]] & =[[x,y],z]  +(-1)^{|x||y|}[y,[x,z]] && \text{(super Jacobi identity).}\label{e:Jacobi}
\end{align}
A bilinear form  $b$ on a Lie superalgebra $\mathfrak{g}$ is called invariant if $b([x,y],z) = b(x,[y,z])$ for all elements $x,y,z\in\mathfrak{g}$.

For a superspace $\ca V$, the general linear Lie superalgebra $\fr{gl}(\ca V)$ is formed by equipping
the associative superalgebra $A = \End(\ca V)$  with the supercommutator
\begin{equation}
	\label{e:supercomm}
	[a,b] = ab - (-1)^{|a||b|} ba\p,
\end{equation} 
for $a,b\in A$ homogeneous elements for the $\bZ_2$-grading on $ A$. 
Denote $\bC^{m|n}$ for the superspace $\ca V =\ca  V_{\bar 0 } \oplus\ca  V_{\bar 1}$ with $\ca  V_{\bar 0 } = \bC^m$ and $\ca  V_{\bar 1} = \bC^n$. In this case, the notation $\mathfrak{gl}(m|n)$ is used instead of $\mathfrak{gl}(\ca V)$.

A bilinear form $b$  on a superspace   $\ca V = \ca V_{\bar 0 } \oplus \ca V_{\bar 1}$ 
is called consistent or even, if for $i,j\in\mathbb{Z}_2$, one has $b(\ca V_i ,\ca V_j ) = 0$ unless $i+ j =\bar{0}$. A consistent bilinear form $b$ is said to be supersymmetric (resp.\ skew-supersymmetric), if $b|_{\ca V_{\bar 0 }\times \ca V_{\bar 0 }}$
is
symmetric (resp.\ skew-symmetric) and $b|_{\ca V_{\bar 1}\times \ca V_{\bar 1}}$
is skew-symmetric (resp.\ symmetric). The subspace where $b$ restricts to a skew-symmetric form necessarily has even dimension.

If $b$ is a non-degenerate, consistent, supersymmetric or skew-supersymmetric, bilinear form on $\ca V$, the
orthosymplectic Lie superalgebra $\mathfrak{osp}(b)=\mathfrak{osp}(\ca V,b)$ is the
subalgebra of $\mathfrak{gl}(\ca V)$ that
preserves $b$. 

As Lie superalgebras $\mathfrak{osp}(\ca V,b)\cong\mathfrak{osp}(\Pi(\ca V),\Pi (b))$, where $\Pi (b)$ is the form on $\Pi(\ca V)$ induced by means of the parity reversing functor $\Pi$. If $b$ is supersymmetric, then $\Pi (b)$ is skew-supersymmetric and vice versa.

For $\ca V = \bC^{M|2n}$ with the standard supersymmetric form, the associated orthosymplectic Lie superalgebra is denoted by $\mathfrak{osp}(M|2n)$ or $\mathfrak{osp}^{sy}(M|2n)$.

For $\ca V = \bC^{2n|M}$ with a skew-supersymmetric form, the associated orthosymplectic (or symplectico-orthogonal) Lie superalgebra is 
$\mathfrak{spo}(2n|M)$ or $\mathfrak{osp}^{sk}(2n|M)$. 
Given the isomorphism between them via $\Pi$, the notation $\mathfrak{osp}(M|2n)$ is sometimes used to refer to either of them.

\subsection{\texorpdfstring{$\fr{osp}(1|2)$}{osp(1|2)}}

The Lie superalgebra $\mathfrak{osp}(1|2)$ has a one-dimensional Cartan subalgebra $\mathfrak{h} = \langle h \rangle$ and root system  $\Phi = \{\pm 2 \delta \} \cup \{ \pm \delta \}$, where $\delta\in \mathfrak{h}^*$ is the dual of $h$.
The even subalgebra is $\mathfrak{osp}(1|2)_{\bar 0} \cong \fr{sp}(2)\cong \fr{sl}(2)$ with root system $\{\pm 2 \delta \}$.


The odd root vectors $e_{\delta},e_{-\delta}$ satisfy the relations
\begin{equation}\label{e:osp12r}
	[ e_{\delta} , 	e_{-\delta} ] = ( e_{\delta},e_{-\delta}) h_{\delta} \p,\qquad [h,	e_{\pm\delta}  ] = \pm e_{\pm\delta}\p,
\end{equation}
where $h_{\delta} = (\delta,\delta)h$ is the coroot of $\delta$. Here, $(\cdot,\cdot)$ denotes the unique (up to a constant factor) 
non-degenerate consistent invariant  supersymmetric bilinear form on $\mathfrak{g}$, and also  the (symmetric) non-degenerate restriction of the form to $\mathfrak{h}$, and the induced form on $\mathfrak{h}^*$. 

With the following normalization for the root vectors
\begin{equation}\label{e:osp12n}
	F^{\pm} \colonequals e_{\pm\delta}/\sqrt{(e_{\delta},e_{-\delta})(\delta,\delta)}\p, \qquad 	
	E^{\pm} \colonequals \pm [e_{\pm\delta}, e_{\pm\delta} ] / (2( e_{\delta},e_{-\delta})(\delta,\delta) )\p,
\end{equation}
and denoting also $H=h$, we have $\mathfrak{osp}(1|2) = \Span\{H,F^\pm,E^\pm\}$ with the relations that have nonzero right-hand side given by
\begin{equation}\label{e:osp12re}
	\begin{aligned}[]
		[ F^{+} , 	F^{-} ] & = H	\p,\quad
		&	[ H , 	F^{\pm} ] &= \pm F^{\pm}\p,
		&	
		[ F^{\pm} , 	F^{\pm} ] & = \pm 2\,E^{\pm}\p,
		\\	
		[ E^{+} , 	E^{-} ] & = H\p,
		&	[ H , 	E^{\pm} ] & = \pm2\,E^{\pm}\p,	\quad	&
		[ 	F^{\pm},E^{\mp}  ] & = 	F^{\mp}\p.	
	\end{aligned}
\end{equation}
The even subalgebra  is $ \Span\{H,E^\pm\} \cong \fr{sl}(2)$.

\subsection{Realization and centralizer}

Let $A = A_{\bar0}  \oplus A_{\bar1} $ be an associative unital superalgebra. 
Using the supercommutator~\eqref{e:supercomm}, if there are elements $e_{\delta},e_{-\delta}\in A_{\bar1}$ satisfying 
\begin{equation}\label{e:osp12}
	[ 	[ e_{\delta} , 	e_{-\delta} ] , 	e_{\pm\delta} ] = \pm C\, e_{\pm\delta}\p,
\end{equation}
for a non-zero constant $C$, then there is a realization of $\mathfrak{osp}(1|2)$ in $A$. The constant $C$ is related to the bilinear forms on $\mathfrak{osp}(1|2)$ and on $\mathfrak{h}^*$
by $C=( e_{\delta},e_{-\delta})(\delta,\delta)$. The elements $e_{\delta},e_{-\delta}\in A_{\bar1}$ can be rescaled as in~\eqref{e:osp12n} to have the commutation relations~\eqref{e:osp12re}. 


Now, assume that we have a realization $\pi \colon \mathfrak{osp}(1|2) \to A$ of $\mathfrak{osp}(1|2)$ in $A$, with the product given by the supercommutator~\eqref{e:supercomm} in $A$. 

For a subspace $ B \subset  A$, the supercentralizer in $ A$ is
\begin{equation}
	\label{e:centralizer}
	\Cent_{ A}( B) = \{\, a \in  A \mid [a,b] = 0 \text{ for all } b\in  B \,\} \rlap{\,.}
\end{equation}
The idea is to describe the supercentralizer of $\mathfrak{osp}(1|2)$ in $A$. 
Hereto, we consider an adjoint action of $\mathfrak{osp}(1|2) $ on $ A$: 
\begin{equation}\label{e:ad}	
		\mathfrak{osp}(1|2) \times A \to  A \colon g \mapsto   
		 (g,a) \mapsto [\pi(g),a]  \p.
\end{equation}
For the action~\eqref{e:ad} of $\mathfrak{osp}(1|2) $, in the representation space $ A$, every element of the centralizer $\Cent_{ A}(\mathfrak{osp}(1|2))$ is a copy of the one-dimensional, trivial module. Conversely, the isotypic component of the trivial module is precisely $\Cent_{ A}(\mathfrak{osp}(1|2))$.

The following proposition gives a way to determine $\Cent_{ A}(\mathfrak{osp}(1|2))$ from the centralizer of the even subalgebra $\Cent_{ A}(\mathfrak{osp}(1|2)_{\bar 0})$. 
To this end, we consider the following (even) elements in the universal enveloping algebra 
$ U(\fr{osp}(1|2))$: 	
\begin{equation}\label{e:Pdelta}
	P_+ \colonequals 1 -  F^-F^+	 \quad\text{ and }\quad 	P_{-} \colonequals 1 + F^+F^-\p,
\end{equation}
where we use the normalization~\eqref{e:osp12n}. 

\begin{propo}\label{p:osp12}
	For $\mathfrak{osp}(1|2)$, realized in an associative unital superalgebra $ A$, with  even subalgebra  $ \mathfrak{osp}(1|2)_{\bar0} \cong\fr{sl}(2)$,
	one has 
	\[
	\Cent_{ A}(\mathfrak{osp}(1|2)) = P_{+} \Cent_{ A}(\mathfrak{osp}(1|2)_{\bar 0}) = P_{-} \Cent_{ A}(\mathfrak{osp}(1|2)_{\bar 0})\p,
	\]
	with $P_\pm$ acting on $ A$ as in~\eqref{e:ad}, that is
	\[
	P_{\pm} \colon  A \to  A \colon a \mapsto P_{\pm}(a)=a\mp [ F^{\mp} , [ F^{\pm}, a]]	\p.
	\]
\end{propo}
\begin{proof}
	First, note that, by means of the relations~\eqref{e:osp12re}, 
	\begin{equation}\label{e:P-dPd}
		P_{-} = 1 + F^+F^-  = 1 + (-F^-F^+ + H)  = P_+ + H \p,
	\end{equation}
	so when $[H,\cdot]$ acts by zero, as is the case on  $\Cent_{ A}(\mathfrak{osp}(1|2)_{\bar 0})$, the actions of $P_{-}$ and  $P_+$ coincide.
	
	By definition, we have 
	$
	\Cent_{ A}(\mathfrak{osp}(1|2)) \subset \Cent_{ A}(\mathfrak{osp}(1|2)_{\bar 0})
	$
	and $
	P_+ a = a
	$ for  $a \in \Cent_{ A}(\mathfrak{osp}(1|2))$,
	hence 
	$
	\Cent_{ A}(\mathfrak{osp}(1|2)) \subset P_+\Cent_{ A}(\mathfrak{osp}(1|2)_{\bar 0})
	$. 
	
	To prove the other inclusion, let $a \in \Cent_{ A}(\mathfrak{osp}(1|2)_{\bar 0})$, we show that $P_+(a) \in \Cent_{ A}(\mathfrak{osp}(1|2))$. We have, using the super Jacobi identity~\eqref{e:Jacobi}, 
	\begin{align*}
		[F^+,	P_+(a)]
		& = [F^+,a] - [F^+,[F^-,[F^+,a]]]\\
		& = [F^+,a] - [[F^+,F^-],[F^+,a]]+ [F^-,[F^+,[F^+,a]]]\\
		& = [F^+,a] - [H,[F^+,a]]+ [F^-,[E^+,a]]\\
		& = [F^+,a] - [[H,F^+],a]- [F^+,[H,a]]\\
		& = [F^+,a] - [F^+,a]\p.
	\end{align*}
	Meanwhile, for $a \in \Cent_{ A}(\mathfrak{osp}(1|2)_{\bar 0})$, we also have
	\begin{align*}
		[F^-,	P_+(a)]
		& = [F^-,a] - [F^-,[F^-,[F^+,a]]]\\
		& = [F^-,a] + [E^-,[F^+,a]]\\
		& = [F^-,a] + [[E^-,F^+],a]+ [F^+,[E^-,a]]\\
		& = [F^-,a] - [F^-,a]\p.
	\end{align*}
	As $F^\pm$ generate $\mathfrak{osp}(1|2)$, this proves the other inclusion $ P_+\Cent_{ A}(\mathfrak{osp}(1|2)_{\bar 0})\subset
	\Cent_{ A}(\mathfrak{osp}(1|2)) 
	$.
\end{proof}

The next result gives a way to obtain, under certain conditions, the centralizer $\Cent_{ A}(\mathfrak{osp}(1|2)_{\bar{0}})$, which can then be used to describe the supercentralizer $\Cent_{ A}(\mathfrak{osp}(1|2))$ by the previous result.

\begin{propo}\label{p:sl2}
	Let $\mathfrak{sl}(2)$ be realized in a unital associative algebra $ A$
	by the elements $e_{\alpha},e_{-\alpha},h_{\alpha}$ satisfying the commutation relations 
	\[
	[e_{\alpha},e_{-\alpha}] = h_{\alpha}\p, \qquad [ h_{\alpha},e_{\pm\alpha} ]=\pm(\alpha,\alpha) e_{\pm\alpha}\p.
	\]	
	If $\Cent_{ A}(\fh)$ decomposes into only finite-dimensional irreducible $\mathfrak{sl}(2)$-modules for the adjoint action of $\mathfrak{sl}(2)$ on $ A$, then 
	\[
	\Cent_{ A}(\mathfrak{sl}(2)) = P_{\alpha} \Cent_{ A}(\fh) = P_{-\alpha} \Cent_{ A}(\fh) \p.
	\]
	where 
	\begin{equation}\label{e:Palpha}
		P_{\alpha} = \sum_{k \geq 0} \frac{(-2)^k e_{-\alpha}^k e_{\alpha}^k}{(\alpha,\alpha)^k k!(k+1)!} \p,
	\end{equation}
	with the elements of $\mathfrak{osp}(1|2)$ acting on $\ca{A}$ via the action~\eqref{e:ad}.
\end{propo}
\begin{proof}

	The elements of $\Cent_{ A}(\fh)$ form the space of weight zero for the action of $\mathfrak{sl}(2)$ given by~\eqref{e:ad}.
	Assume $\Cent_{ A}(\fh)$ decomposes into only finite-dimensional irreducible $\mathfrak{sl}(2)$-modules. 
	For $v\in \Cent_{ A}(\fh)$, 
	the sum in~\eqref{e:Palpha} reduces to a finite one when acting on $v$, and as
	$e_{\alpha} P_{\alpha}v = 0 $, which follows via 
	\begin{align*}
		e_{\alpha}	e_{-\alpha}^k e_{\alpha}^k =&\	e_{-\alpha}^k e_{\alpha}^{k+1} + k\, e_{-\alpha}^{k-1}(h_{\alpha}-(\alpha,\alpha)(k-1)/2) e_{\alpha}^k 	\\
		=&\	e_{-\alpha}^k e_{\alpha}^{k+1} + k\, e_{-\alpha}^{k-1} e_{\alpha}^k(h_{\alpha}+(\alpha,\alpha)(k+1)/2) \p,
	\end{align*}
	so $P_{\alpha} v$  is a highest weight vector of a finite-dimensional $\mathfrak{sl}(2)$-module with weight zero, hence a trivial module.
\end{proof}	

\begin{remar}\label{r:proj}
	The formulas~\eqref{e:Pdelta} and~\eqref{e:Palpha} correspond to a so-called ``extremal projector'', see for instance~\cite{AshSmiTol79,Tolstoi1985,Zhelobenko}, when acting on a space of weight zero.
	
	For $\fr{g}$ a basic classical Lie (super)algebra (or generalization thereof), the extremal projector $P$ is an element of an extension of the universal enveloping algebra $U(\fr{g})$ containing formal power series with coefficients in the field of fractions of $U(\mathfrak{h})$, also called a transvector algebra or Mickelsson-Zhelobenko algebra (localization with respect to $\fh$) \cite{Zhelobenko,RE}. The element $P$ is the unique nonzero solution to the equations
	\begin{equation}\label{e:exproj}
		P^2 = P, \qquad e_{\alpha} P = 0 = P e_{-\alpha} \p,\quad  \text{for all positive roots }\alpha\text{ of } \fr{g}\p.
	\end{equation}
	Hence, when acting on a representation space it would project onto highest weight vectors.
	
	However, for Lie superalgebras with isotropic roots, when acting on a space of weight zero the formula for the extremal projector $P$ can result in denominators becoming zero.
\end{remar}

We now consider a lemma with some properties of $P_\pm$ that will be used in Section~\ref{s:osp12}. Similar properties hold for $P_{\alpha}$.

\begin{lemma}\label{l:Pdelta}	For a realization of $\mathfrak{osp}(1|2)$ in  $ A$, and $P_{\pm}$ given by~\eqref{e:Pdelta}, let $a,b,c\in  A$, then 
	\[
	P_{\pm} (a+b) = P_{\pm} (a)+ P_{\pm} ( b) \rlap{\,.}
	\]
	If $a\in\Cent_{  A}(\mathfrak{osp}(1|2))$, then 
	\[
	P_{\pm} (a) = 	a\rlap{\,.}
	\]	
	If $a\in\Cent_{  A}(\mathfrak{osp}(1|2))$ or $b\in\Cent_{  A}(\mathfrak{osp}(1|2))$, then 
	\[
	P_{\pm}  (ab) = 	P_{\pm}  (a)P_{\pm} (b)\rlap{\,.}
	\]	
	If  $a,c\in\Cent_{  A}(\mathfrak{osp}(1|2))$, then 
	\[
	P_{\pm}  (abc) = 	aP_{\pm} (b)c\rlap{\,.}
	\]
\end{lemma}
\begin{proof}
	The first relation follows immediately from the bilinearity of the supercommutator and the second relation from the 	
	 definition of $\Cent_{ A}(\mathfrak{osp}(1|2))$.
	
	For the third result, assume $a$ is a homogeneous element for the $\bZ_2$-grading on $A$. With the  action of $\mathfrak{osp}(1|2)$ on $A$ given by~\eqref{e:ad}, we have
	\[
	P_{+}(ab)
	= 	ab -  F^-F^+(ab)
	=   ab -  F^-(F^+(a)b + (-1)^{|a|}aF^+(b))
	\p.
	\]
	On the one hand, if $a\in\Cent_{A}(\mathfrak{osp}(1|2))$, then $F^{\pm}(a)=0$, so this becomes	
	\[
P_{+}(ab)
=   ab - (-1)^{|a|} (F^-(a)F^+(b) + (-1)^{|a|}a (F^-F^+b))
= a P_{+}(b)
\p.
\]
	On the other hand, if $b\in\Cent_{  A}(\mathfrak{osp}(1|2))$, so $F^{\pm}(b)=0$, this becomes	
	\begin{align*}
		P_{+}(ab)
		& =  ab -   (F^-F^+a)b  = P_+(a) b\p.
	\end{align*}
	
	The final relation follows in the same manner.
\end{proof}

\begin{examp}
	In the universal enveloping algebra $ U(\fr{osp}(1|2))$, the
	centralizer of the even subalgebra is generated by the	
	 element $F^+F^--F^-F^+ $ and the constants.  
	 This follows by classical invariant theory.
	Indeed, using~\eqref{e:osp12re}, we have
	\begin{align*}
		[E^\pm, F^+F^--F^-F^+ ] = \mp (F^\pm)^2 \pm (F^\pm)^2= 0 \p.
	\end{align*}

	Now, applying $P_+$ and using~\eqref{e:osp12re}, we find
	\begin{align*}
		P_+(F^+F^-   ) 
		& =  F^+ F^-   - [F^-,[F^+,F^+ F^-  ]] \\
		& =  F^+ F^-  
		- [F^-,2E^+F^- 
		- F^+ H]\\
		& =  F^+ F^- 
		- 2F^+F^- 
		+4E^+E^-
		+H^2
		-F^+  F^- \\
		& =  H^2
		+4E^+E^-
		-2F^+  F^- 
		\p,
	\end{align*} 
	with a similar expression for 
	\begin{align*}
		P_+(F^-F^+   ) 
		& =  F^- F^+   - [F^-,[F^+,F^- F^+  ]] \\
		& =  F^- F^+  
		- [F^-,HF^+ 
		- 2F^- E^+]\\
		& =  F^- F^+ 
		-F^-F^+ - H^2 -4E^-E^+ - 2F^-F^+
		\\
		& =
		- H^2 -4E^-E^+  -2F^- F^+ 
		\p.
	\end{align*} 
	Combining the two yields that $P_+(F^+ F^-  -  F^-F^+   )$ is  proportional to the quadratic Casimir element of $ U(\fr{osp}(1|2))$:
	\begin{equation}\label{e:Casiosp}
		\Omega_{\fr{osp}} = H^2
		+2(E^+E^-+E^-E^+)
		-(F^+  F^- -F^-F^+) \p,
	\end{equation} 
	where $\Omega_{\fr{sl}(2)} = H^2
	+2(E^+E^-+E^-E^+)$ is the quadratic Casimir element of $U(\fr{sl}(2))$. 
	Note that $F^+ F^-  -  F^-F^+$ is related to the $\fr{osp}(1|2)$ Scasimir element~\cite{Casi}:
	\begin{equation}\label{e:SCasiosp}
		S = F^+F^- -F^-F^+ + 1/2 \p,
	\end{equation}
	which squares to $S^2 = \Omega_{\fr{osp}} + 1/4$. 
	The Scasimir element $S$ 
	commutes with the even elements and anticommutes with the odd elements, so it antisupercommutes with $ U(\fr{osp}(1|2))$ since its parity is even.
	By the above we also have 
	\begin{equation}\label{e:SCasiosprel}
		P_\pm(S) = 2\Omega_{\fr{osp}} + 1/2 = 2S^2\p.
	\end{equation}
\end{examp}

\subsection{Generalized symmetries}\label{s:gensyms}

Let $A$ be an associative unital (super)algebra. 
We say that an element $a\in A$ is a generalized symmetry of $F\in A$ if there exists $b\in A$ such that $F\, a = b\, F$. Note that $a$ preserves the kernel of $F$. 

Extremal projectors and transvector algebras can be used to construct generalized symmetries of either the positive or negative root vectors of a Lie (super)algebra realized in $A$, see also \cite{Zhelobenko,RE}. 
Note that the extremal projectors can contain 
fractions of $U(\mathfrak{h})$, so depending on the algebra $A$ one works in, a multiplication by elements of $U(\mathfrak{h})$ can be required to cancel denominators. 

We now consider the case of $\fr{osp}(1|2)$, which we will use in Section~\ref{s:gensyms2} for an explicit realization. 
Define 
\begin{equation}\label{e:Q}
	Q^{\pm} \colon \ca A \to \ca A \colon a \mapsto Q^{\pm}(a)=(H\pm1)a\mp F^{\mp}  [ F^{\pm}, a]	\p.	
\end{equation}

\begin{propo}\label{p:gensym}
Let $a \in A$ be such that $[E^-,a] = b F^-$ for some $b \in A$, then
\[
Q^-(a) = (H-1)a - F^+ [F^-,a] 
\]
is a generalized symmetry of $F^-$.
\end{propo}
\begin{proof}
 Assume $a$ is homogeneous for the $\bZ_2$-grading
\begin{align*}
	F^-Q^-(a)  & = F^-(H-1)a - F^-F^+ [F^-,a] \\* 
	& = HF^-a - (H-F^+F^-  )  [F^-,a]    \\*  
	& =(-1)^{|a|} H a F^- + F^+F^- (F^-a  - (-1)^{|a|} a F^- )   \\* 
		& = (-1)^{|a|} H a F^-  - F^+E^- a  - F^+(-1)^{|a|} a F^-   \\* 
	& = (-1)^{|a|} H a F^- - F^+ ( a E^- + bF^-)
	 - (-1)^{|a|} F^+ F^- aF^-   \\* 
	 & = (-1)^{|a|}( H a  + (-1)^{|a|} F^+  a F^- -(-1)^{|a|} F^+b
	 - F^+ F^- a)F^-  \\* 
	 & = (-1)^{|a|}( Q^-(a) + a  -(-1)^{|a|} F^+b   )F^- 
\end{align*}
which shows that $Q^-(a) $ is a generalized symmetry of $F^-$. 
\end{proof}
A similar result holds for $F^+$ using $Q^+$.

More generally, we can take $a\in A$ such that for some $b\in A$
\[
F^- [F^-,a] = F^-(F^-a  - (-1)^{|a|} a F^- ) = b F^-\p,
\]
which is the case if for some $c\in A$
\[
F^-F^-a  = c F^-\p.
\]

\section{Dunkl realization} \label{s:Dd}

We consider a complex vector space $ V\cong \mathbb{C}^{d}$, for a positive integer $d$. 
The ring of polynomial functions on $V$ is the symmetric algebra $S(V^*)$ of the dual space $V^*$. 
In the next sections our focus will be on the dual space $V^*$.  
The notation $\langle \cdot,\cdot \rangle$ will denote the natural bilinear pairing between a space and its dual. 

\subsection{Bilinear form}\label{s:bilform}
Let $B$ denote a non-degenerate symmetric bilinear form on $V$. 
The orthogonal group $\rO \colonequals\rO(V,B) \cong\rO(d,\bb C)$ is the group of invertible linear transformations of $V$ that preserve the form $B$. The action of $\rO$ on $V$ is naturally extended to $V^*$ as the contragradient action, and in turn also to tensor products of those spaces. 
 There is an isomorphism $\beta \colon V \to V^*\colon v \mapsto \beta(v) 
 $ given by
\begin{equation}\label{e:beta}
	\langle \beta(u_1) , u_2 \rangle
	 = B(u_1,u_2)\p, \quad\text{for } u_1,u_2\in V\p,
\end{equation}
which commutes with the action of $\rO$. 
Hence, the spaces $V$ and $V^*$ are isomorphic as $\rO$-modules. 
We will use $\beta$ also to denote its inverse so $\beta$ becomes an involution of $V\oplus V^*$. Moreover, we will denote by $B$ also the induced bilinear form on $V^*$ given by $B(u,v) = B(\beta(u),\beta(v))$ for $u,v\in V^*$.

\subsection{Bases}\label{s:bases}
If no specific properties are needed, $v_1^*,\dotsc,v_d^*$ will denote a basis of $V^*$, dual to a basis $v_1,\dotsc,v_d$ of $V$,  so $ \langle v^*_p,v_k\rangle = \delta_{p,q}$ for $p,q\in \{1,\dots,d\}$. 
In terms of these bases, for $u \in V^*$ and $v\in V$ we have
\begin{equation}\label{e:B}
	\beta(v) = \sum_{p=1}^d B( v, v_p) v^*_p\p, \qquad
	\beta(u) = \sum_{p=1}^d B( u, v^*_p) v_p\p,
\end{equation}
and
\begin{equation}\label{e:BB}
	v  = \sum_{p=1}^d B( v, v_p) \beta(v^*_p)
	\p, \qquad 	
	u = \sum_{p=1}^d B( u, v^*_p)  \beta(v_p)
		\p.
\end{equation}
so 
\begin{equation}\label{e:BBB}
	v  = \sum_{p,q=1}^d  B(v,v_p) B(v_p^*,v_q^*) v_q 
		\p, \qquad 	
		u =\sum_{p,q=1}^d  B(u,v_p^*) B(v_p,v_q) v_q^* 
	\p.
\end{equation}

For brevity, we will sometimes denote $B_{pq} = B(v_p,v_q)$.

When needed, $y_1,\dotsc,y_d$ will denote a basis of $V$ and  $x_1,\dotsc,x_d\in V^*$ its dual basis, such that $\delta_{j,k} =\langle x_j,y_k\rangle =  B(x_j,x_k) = B(y_j,y_k)$. 
Note that the involution $\beta$ sends $y_j$ to $x_j$ and vice versa. 

Let $V^*=V^+\oplus V^0 \oplus V^-$ be a Witt decomposition of $V^*$ where $V^+$ and $V^-$ are complementary maximal $B$-isotropic subspaces of $V^*$ of dimension $\ell = \lfloor d/2 \rfloor$, and $V^0 = \emptyset$ for $d$ even while $V^0$ is anisotropic and one-dimensional for $d$ odd. 
Let $z_1^+,\dotsc,z_\ell^+$ denote a basis of $V^+$ and $z_1^-,\dotsc,z_\ell^-$ a basis of $V^-$ such that $B(z_j^+,z_k^-) = \delta_{j,k}/2$. 
For $d$ odd, denote $z_0$ an element of $V^0$ satisfying $B(z_0,z_0) = 1$. 
The dual basis of $V$ is given by $w_j^{\pm} \colonequals 2\beta(z_j^\pm)$, satisfying $\langle z_j^\pm , w_j^\mp \rangle = \delta_{j,k}$. 

\subsection{Clifford algebra}

As our focus will be on $V^*$, we will construct the Clifford algebra associated with $V^*$ and $B$. Note that this results in the same Clifford algebra as when using $V$ and $B$. 

The Clifford algebra $\cC \colonequals \cC(V^*,B)$ is
the quotient of the tensor algebra $T(V^*)$ by the ideal generated by all elements of the form $v\otimes v  -  B(v,v)1 $ for $v\in V^*$. 
The quotient map from the embedding $V^* \to T(V^*)$ gives a canonical map $\gamma \colon V^* \to \cC$.

We have that $\cC$ is the associative algebra with 1 generated by $\gamma(V^*)$, subject to the anticommutation relations
\begin{equation}\label{e:Clifcom}
	\gamma(u)\gamma(v)+\gamma(v)\gamma(u) = 2 B(u,v)\,1\qquad\text{for }u,v\in V^*\p.
\end{equation}
The Clifford algebra inherits the structure of a filtered super algebra from $T(V^*)$, with the generators $\gamma(V^*)$ being odd and having filtration degree 1. 

For $u\in V^*$, we will denote $\gamma_u \colonequals \gamma(u) \in \cC$, and for $u_1,\dotsc,u_k\in V^*$ we let $\gamma_{u_1\dotsm u_k} \colonequals \gamma_{u_1}\dotsm \gamma_{u_k} \in \cC$. 
Moreover, we denote 	$\gamma_{u_1\dotsm \widehat{u}_j\dotsm u_n}\colonequals	\gamma_{u_1} \dotsm \widehat{\gamma}_{u_j} \dotsm \gamma_{u_n}$ where the notation $ \widehat{\gamma}_{u_j}$ indicates that the factor $ \gamma_{u_j}$ is omitted in the product.

For the chosen bases of $V^*$ we denote $e_j = \gamma(x_j)$ for $j\in\{1,\dotsc,d\}$, and 
$\theta^\pm_k = \gamma(z^\pm_k)$ for $k\in \{1,\dotsc,\ell\} $ with $\theta_0 = \gamma(z_0)$ for $d$ odd.
In $\ca C$, they satisfy the relations
\begin{equation}
 e_je_k+e_ke_j = 2\delta_{j,k} \p, \quad \theta_j^\pm\theta_k^\pm + \theta_k^\pm\theta_j^\pm = 0\p, \quad \theta_j^+\theta_k^- + \theta_k^+\theta_j^- = \delta_{j,k}\p,
\end{equation}
and when $d$ is odd, also $\theta_0^2 = 1$ and $\theta_0\theta_k^\pm + \theta_k^\pm\theta_0 = 0$. 

For a subset  $A \subset \{1,\dotsc,d\}$, with elements $A= \{ a_{1},a_{2},\dotsc,a_{n}\} $ such that $1\leq a_{1}<a_{2}<\cdots <a_{n}\leq d$, we denote
$
e_A = e_{a_{1}}e_{a_
	{2}}\cdots e_{a_{n}}
$.
Let $e_{\emptyset} = 1$, then a basis for $\ca C$ as a vector space is given by $\{e_A \mid A \subset \{1,\dotsc,d\} \}$.

We denote the chirality element of the Clifford algebra as
\begin{equation}\label{e:Gamma}
	\Gamma \colonequals i^{d(d-1)/2} e_1 \dotsm e_d \in \ca C;
\end{equation}
it satisfies $\Gamma^2 = 1$ and $\Gamma \gamma_u = (-1)^{d-1} \gamma_u \Gamma$ for $u\in V^*$.

Let $\ca A$ denote the anti-symmetrization operator, which has the following action on a multilinear expression in $n$ indices
\begin{equation}\label{e:asym}
	\ca A(	f_{u_1u_2\dotsm u_n})= \frac{1}{n!} \sum_{s \in \rm{S}_n} \sgn (s) f_{u_{s(1)} \dotsm u_{s(n)}} \p,
\end{equation}
where $\rm{S}_n$ is the symmetric group of degree $n$.	
We have $\ca A \ca A = \ca A$ and 
\begin{equation}\label{e:asym2}
	\ca A(	f_{u_1u_2\dotsm u_n})= \frac{1}{n} \sum_{j=1}^n (-1)^{j-1} f_{u_j} 	\ca A(	{}_{u_1\dotsm \widehat{u}_j \dotsm u_n}) \p.
\end{equation}	

The symbol map and its inverse, the quantization map
\begin{equation}\label{e:quant}
q \colon \bigwedge( V^*) \to \ca C \colon u_1 \wedge u_2 \wedge \dotsb \wedge u_k \mapsto \cA{u_1\dotsm u_k}\p,
\end{equation}	
are isomorphisms of $\rO$-modules and filtered super vector spaces~\cite[Section 2.2.5]{Meinrenken}.

For $u,v,w,x\in V^*$, we have
\begin{align}
	\ca A (\gamma_{uv}) &  =\gamma_{uv}-B(u,v)\label{e:Auv}\\
	\ca A (\gamma_{uvw}) &=\gamma_{uvw}-B(u,v)\gamma_w  +B(u,w)\gamma_v -B(v,w)\gamma_u \label{e:Auvw}\\ 
	\ca A (\gamma_{uvwx})& =\gamma_{uvwx} -B(u,v)\gamma_w\gamma_x  +B(u,w)\gamma_v\gamma_x -B(v,w)\gamma_u \gamma_x\label{e:Auvwx}\\
	& \quad
	-B(u,x)\gamma_v\gamma_w  +B(v,x)\gamma_u\gamma_w -B(w,x)\gamma_u \gamma_v\notag\\
	& \quad
	+B(u,v)B(w,x)  -B(u,w)B(w,x) +B(v,w)B(w,x) \notag\p.
\end{align}
Note that if $u_1,\dotsc,u_n\in V^*$ are $B$-orthogonal, then 
\[
\ca A(\gamma_{u_1u_2\dotsm u_n})= 	\gamma_{u_1u_2\dotsm u_n}\p.
\]	

With the commutator, the space $q(\wedge^2(V^*))$ in $\ca C$ forms a realization of the Lie algebra $\fr{so}(V^*,B) \cong \fr{so}(V,B) \cong \fr{so}(d,\bC)$. We have the following adjoint action on $\gamma(V^*)$: 
\begin{equation}\label{e:adso}
	[\gamma_{uv}/2,\gamma_w] =		\gamma_u[\gamma_v,\gamma_w]/2 - 	[\gamma_u,\gamma_w]\gamma_v/2 = B(v,w) \gamma_u- 	B(u,w)\gamma_v \p.
\end{equation}

\subsection{Spinor space}

When $d$ is odd, there is a unique isomorphism class of irreducible $\bZ_2$-graded $\ca C$-modules, and there are two isomorphism classes of irreducible ungraded $\ca C$-modules, see \cite[Theorem~3.10]{Meinrenken}. 
When $d$ is even, there are two isomorphism classes of irreducible $\bZ_2$-graded $\ca C$-modules, and there is a unique isomorphism class of irreducible ungraded $\ca C$-modules.

A model for an irreducible $\bZ_2$-graded $\ca C$-module $\bb S$ is given by the exterior algebra $ \bigwedge (V^+ \oplus V^0)$, where $V^*=V^+\oplus V^0 \oplus V^-$ is a Witt decomposition of $V^*$. 
On the space $\bb S$ the elements of $\gamma(V^+)$ act by exterior multiplication and those of $\gamma(V^-)$ by interior multiplication or contraction. 
In terms of a $B$-isotropic basis of $V^*$ as defined in Section~\ref{s:bases}, 
$\theta^-_{j}\in \gamma(V^-)$ acts as the odd differential operator corresponding to the odd variable $\theta^+_{j}\in \gamma(V^+)$. 
 
When $d$ is even, $V^0 = \emptyset$ and this gives the complete action of $\ca C$ on $\bb S$. The parity-reversed space $\Pi ( \bb S ) $ is another irreducible $\bZ_2$-graded $\ca C$-module, they are isomorphic as ungraded $\ca C$-modules. 
The spinor space $\bb S$ can be realized explicitly inside the Clifford algebra as 
\begin{equation}
\bb S =\bigwedge (V^+ ) \prod_{j} (\theta_{j}^- \theta_{j}^+)  \p, 
\end{equation}
with the action of $\ca C$ given by Clifford algebra multiplication, since the product $\prod\theta_{j}^- \theta_{j}^+$ is annihilated by all $\theta^-_{k}$.

When $d$ is odd, recall $\theta_0\in \gamma(V^0)$ with $B(\theta_0,\theta_0)=1$, we can write $\bb S = \bb S_+ \oplus  \bb S_-$ where 
$\bb S_\pm = \bigwedge (V^+) (1\pm\theta_0)/2  $ are non-isomorphic irreducible ungraded $\ca C$-modules. 
The action of $\theta_{0}$ on $S_\pm$ is given by
\begin{equation}\label{e:acttheta0}
\theta_{0} \cdot \theta = \pm (-1)^k \theta \p,\quad \text{ for } \theta \in \textstyle\bigwedge^k (V^+)(1\pm\theta_0)/2\subset \bb S_\pm\p,
\end{equation}
and extending by linearity. The sign in the action~\eqref{e:acttheta0} of $\theta_{0}$ distinguishes between the spaces $\bb S_{\pm}$.
The spinor space $\bb S$ can be realized explicitly inside the Clifford algebra by letting
\begin{equation}
\bb S_{\pm} = \bigwedge (V^+ )(1\pm\theta_0)/2 \prod_{j} (\theta_{j}^- \theta_{j}^+) \p, 
\end{equation}
with the action of $\ca C$ given by Clifford algebra multiplication, since $\theta_0(1\pm\theta_0)/2 = \pm (1\pm\theta_0)/2 $.

%

\subsection{Reflection group}
We fix a finite real reflection group $G\subset \rO$.
Denote by $\mathcal S$ the set of reflections of $G$.
For each $s\in \mathcal{S}$, fix $\alpha_s\in V^*$ to be a $-1$ eigenvector for the action of $s$. 

By definition, $\beta(\alpha_s)\in V$ is a $-1$ eigenvector for the action of $s$ on $V$.
Denote $	\alpha_s^{\vee} 
 \colonequals 2\beta(\alpha_s) / B(\alpha_s,\alpha_s) $, then 
for $v\in V$ and $u\in V^*$, the  reflection $s\in \mathcal{S}$ acts as
\begin{equation}\label{e:s}
	s( v) = v - \langle\alpha_s, v \rangle \alpha_s^{\vee}\p,\qquad 
		s(u) = u - \alpha_s\langle\alpha_s^{\vee}, u \rangle \p.
\end{equation} 

Define $T(V \oplus V^*) \rtimes G$ to be the quotient of $T(V \oplus V^*) \otimes \bC[G]$ by the relations
\begin{equation}\label{e:crossedproduct}
(u,g)(v,h) = (u\,g(v),gh) \qquad \text{for } u,v\in T(V \oplus V^*) \text{ and } g,h\in G\p,
\end{equation} 
so in $T(V \oplus V^*) \rtimes G$, we have $gug^{-1} = g(u)$ for $g\in G$ and $u\in T(V \oplus V^*)$.

We fix a map $\kappa \colon \mathcal S \to \bC $ that is $G$-invariant (for the conjugation action), so that the elements of an orbit all have the same image. 

\begin{defin}
Define $H_\kappa = H_\kappa(G,V)$ to be the quotient of $T(V \oplus V^*) \rtimes G$
by the  relations
\begin{equation}
	\label{e:RC}
	\begin{aligned}{}
		[x,u] &= 0 = [y,v]\p, \qquad\text{for } y,v \in V \text{ and } x,u\in V^*\\
		[y,x] &= \langle y, x\rangle +   \sum_{s\in\mathcal S} \langle y,	\alpha_s\rangle\langle  \alpha_s^{\vee},  x \rangle  \kappa(s) s \p, \quad\text{for } y \in V,x\in V^*\p.
	\end{aligned}
\end{equation} 		
\end{defin}
When $\kappa$ is the zero map, the relations~\eqref{e:RC} reduce to the canonical commutation relations of the Weyl algebra $\ca W = \ca W(V)$, so $ H_0(G,V) = \ca W(V)  \rtimes G $.

\begin{remar}
	The algebra $H_\kappa(G,V)$ is called a rational Cherednik algebra, and is
	a rational degeneration of a double affine Hecke algebra \cite{EtGi,GGOR}. 
		More generally, a complex reflection group $G\subset GL(V)$ can be used.
		Also,  there can be an extra parameter $t$ accompanying  $\langle y, x\rangle$ in \eqref{e:RC}, which we have taken $t=1$ here. See \cite{Kieran} for the case where $t\in \bC^\times$.


	A realization of $H_\kappa$ is given by means of Dunkl operators~\cite{Du}
	\begin{equation}\label{e:dunkl}
		\ca{D}_{y} = \frac{\partial}{\partial y} +  \sum_{s\in\mathcal S}  \kappa(s) \frac{\langle  y,	\alpha_s\rangle}{\alpha_s}    (1-s)\p,\qquad \text{for } y\in V\p,
	\end{equation} 
	and coordinate variables (for the elements of $V^*$), which gives a natural, faithful action on  the polynomial space $S(V^*)$. 
	
	In the context of a rational Cherednik algebra, the parameter function is usually chosen the opposite sign compared to the one used for Dunkl operators, corresponding to the substitution $\kappa = -c$.
\end{remar}

\begin{lemma}\label{l:Buv}
For $u,v\in V^*$ or $u,v\in V$, in  $H_\kappa$  we have $[\beta(u),v]=  [\beta(v), u]$. 
\end{lemma}
\begin{proof}
	Let $u,v\in V^*$, then we have
	\begin{equation}\label{e:combuv}
		[\beta(u),v]= B(u,v) + 2\sum_{s\in\mathcal S} \frac{B( \alpha_s,u)B( v,\alpha_s ) }{B(\alpha_s,\alpha_s)}  \kappa(s) s  = [\beta(v), u]\p,
	\end{equation}  
where the last equality follows from $B$ being symmetric.
\end{proof}

\subsection{Superalgebra}

We consider the superspace $\ca V = \bC^{2|1}$ equipped with a non-degenerate, skew-supersymmetric, consistent, bilinear form $b$.
Denote by $\omega$ the skew-symmetric bilinear form on $\ca V$, and also its restriction to $\ca V_{\bar 0}=\bC^{2}$, that equals $b$ on  $\ca V_{\bar 0} $ and is zero on  $(\ca V  \times \ca V ) \setminus (\ca V_{\bar 0}   \times \ca V_{\bar 0}  )$.
 
The tensor product $\ca U = V^*\otimes \ca V$ is again a superspace, inheriting the $\bZ_2$-grading from $\ca V$, so $\ca U_{\bar 0} = V^*\otimes \ca V_{\bar 0}$ and $\ca U_{\bar 1} = V^*\otimes \ca V_{\bar 1}$. 
There is a natural action of $\rO$ on $\ca U = V^*\otimes \ca V$ as 
$\rO\otimes \Id_{\ca V}$ where $\Id_{\ca V}$ denotes the identity on $\ca V$. For $G \subset\rO$  
we consider also another action on $\ca U$: 
\begin{equation}\label{e:actO}
	a_0 \colon G \times \ca U \to \ca U \colon 
	\begin{cases}
		a_0(g,u\otimes v) = (g\cdot u)\otimes v & \text{ for }u\in V^*, v \in \ca V_{\bar 0} \\
			a_0(g,u\otimes v) =u\otimes v  & \text{ for }u\in V^*, v \in \ca V_{\bar 1}\p,
	\end{cases}
\end{equation}
where $g\cdot u$ denotes the action of $g\in G$ on $u\in V^*$.
The ``missing'' interaction of $G$ and $\ca U_{\bar 1}$ will be provided by means of the $\Pin$-group inside the Clifford algebra, see~\eqref{e:rho}.

The action \eqref{e:actO} is extended naturally to the tensor superalgebra $T(\ca U) = \bigoplus_n \ca U^{\otimes n}$, which uses $\bZ_2$-graded tensor products. Using~\eqref{e:actO}, define $T(\ca U) \rtimes G$ to be the quotient of $T(\ca U) \otimes \bC[G]$ by the relations
\begin{equation}\label{e:crossproduct2}
(u,g)(v,h) = (u\,a_0(g,v),gh) \qquad \text{for } u,v\in T(\ca U), g,h\in G\p.
\end{equation}

Now, we consider the $G$-invariant symmetric bilinear map 
$\psi_\kappa^B(\cdot,\cdot) \colon V^*\times  V^* \to \mathbb{C}[G] $ 
defined as
\begin{equation}\label{e:psiB}
	\psi_\kappa^B(u,v) =  
	2\sum_{s\in\mathcal S} 
	\frac{B( \alpha_s,u)B( v,\alpha_s ) 
	}{B(\alpha_s,\alpha_s)}  \kappa(s) s \p, \qquad\text{for }u,v\in V^*\p.
\end{equation} 	

\begin{defin}
Define the superalgebra $A_\kappa$  to be the quotient of 
$(T(\ca U) \rtimes G) $ by the relations
\begin{equation}\label{e:comre}
	u v - (-1)^{|u||v|} vu = b_{\ca{U}}(u,v)\,1+ \psi_\kappa(u,v)\,1 \qquad\text{for }u,v\in \ca U_{\bar0} \cup \ca{U}_{\bar1}\p,
\end{equation}
where the right-hand side is defined for $u \otimes w, v\otimes z \in V^*\otimes \ca V = \ca U $ as 
\[
b_{\ca{U}}(u \otimes w,v\otimes z) = B(u,v)b(w,z) \p,
\qquad
\psi_\kappa(u \otimes w,v\otimes z) =\psi_\kappa^B(u,v) \omega(w, z)\p.
\]
\end{defin}
The superalgebra $A_\kappa$ is generated, as an algebra, by $\ca U$ and $G$. 
We now show that $A_\kappa$ is the tensor product of the rational Cherednik algebra $H_\kappa$ and the Clifford algebra $\ca C$. 

Fix $x^+,x^-\in \ca V _{\bar 0}$ and $\gamma\in \ca V _{\bar 1}$ to be a basis of $\ca V = \bC^{2|1}$ satisfying $b(x^-,x^+) = 1 = -b(x^+,x^-)$ and $b(\gamma,\gamma)=2$, with $b$ zero for all other combinations. 

For the elements of 
the form $ u \otimes \gamma \in  \ca U$,
we see that the relations~\eqref{e:comre} correspond precisely to the Clifford algebra relations~\eqref{e:Clifcom}. For $u\in V^*$, we will identify $\gamma_u= u \otimes \gamma$. 

Identifying $u\in V^*$ with $u \otimes x^+ \in \ca U$, and $u \otimes x^- \in \ca U$ with $\beta(u) \in V$, 
the relations~\eqref{e:comre} then correspond precisely to~\eqref{e:RC}.
An element $v \in V$ then corresponds to $\beta(v)\otimes x^- \subset \ca U$.


\subsection{Double cover} 

The group $\Pin \colonequals \Pin(V,B)$ is a double cover  $p\colon\Pin\to\rO$  and is realized in the Clifford algebra $\ca C$ 
as the set of products $\gamma_{u_1\dotsm u_k}$ where $u_j\in V^*$ with $B(u_j,u_j) =1$~\cite{Meinrenken}. The subgroup $\Spin(V,B)$ consists of similar products with $k$ even.

For a reflection $s\in \ca S \subset \rO$, denote $\tilde s \colonequals \gamma(\alpha_s)/\sqrt{B( \alpha_s,\alpha_s)} \in \Pin\subset \ca C$, then $p(\widetilde s) = s$ and $p^{-1}(s) =\{ \pm\tilde s\}\subset \ca C$.
The preimage of the identity $\Id \in \rO$ is  $p^{-1}(\Id) =\{ \pm1\}\subset \ca C$. 

Define the pin double cover of $G\subset \rO$ as $\widetilde G \colonequals p^{-1}(G) \subset \Pin$. The conditions for $\widetilde G$ to be a non-trivial central extension of $g$ are in~\cite{Morris}. 
See~\cite{Morris} also for a presentation in terms of generators and relations for $G$ and $\widetilde G$.

\begin{remar}
	Note that this is the version of the $\Pin$-group, and thus of $\widetilde G$, where the preimages (for the covering map $p$) of a reflection in $\rO$ have order two (and not four).	
	The (non-isomorphic) other version can be obtained by using elements $u_j\in V^*$ with $B(u_j,u_j) =-1$, or by adding a minus sign to the defining relations of the Clifford algebra.
	We refer to~\cite[Section~3.7.2]{Meinrenken} for the definition and the distinction with the group $\Pin_c$. 
\end{remar}

In the superalgebra $A_\kappa \cong H_\kappa \otimes \ca C$, there is a copy of the group $G\subset H_\kappa$ and also of the group $	\widetilde{G} \subset \ca C$. We use these to define a group homomorphism 
\begin{equation}\label{e:rho}
	\rho \colon \tilde G \to 	A  \colon \tilde s \mapsto p(\tilde s) \, \tilde s\p,
\end{equation}
which is extended linearly to a map on the group algebra $\mathbb{C}[\widetilde G]$. 
We note that $\rho(\mathbb{C}[\widetilde G])$ is a quotient of the group algebra $\mathbb{C}[\widetilde G]$, since the central element (the non-trivial preimage of the identity) is given by the scalar $-1\in\bC$
for the realization of $\widetilde{G}$ in  $\Pin \subset\ca C$, see~\cite{Kieran}.

Recall that, as an algebra, the superalgebra $A_\kappa$ is generated by $\ca U = V^*\otimes \ca V$ and $G$. 
\begin{defin}\label{d:actG}
	For $g\in G$ and $u\in A_\kappa$, we denote by $g\cdot u$ the action $G \times A_\kappa \to A_\kappa$ that is the extension of the natural action of $G\subset \rO$ on $V^*$, acting as $G \otimes \Id_{\ca V}$ on $V^* \otimes \ca V$, and of the action by conjugation on the copy $G\subset H_\kappa$.
\end{defin}

This action of $G$ is related to the action of $\rho(\widetilde G)$ inside $A_\kappa$ as follows. 

\begin{lemma}\label{l:rhoG}
For $\tilde g \in \widetilde G$ and $u\in A_\kappa$ a homogeneous element for the $\bZ_2$-grading, in $A_\kappa$ we have
\begin{equation}
	\rho(\tilde g) u \rho(\tilde g^{-1}) =  (-1)^{|\tilde g||u|}p(\tilde g) \cdot u\p.
\end{equation}	
\end{lemma}
\begin{proof}
	Use~\eqref{e:actO}, \eqref{e:rho} and the properties of the $\Pin$-group. 
\end{proof}

Finally, for $u\in V^*$, we define the following elements in $ \rho(\widetilde G)$:
\begin{equation}\label{e:Ov}
	\oO_u  \colonequals \frac12\sum_{s\in\mathcal S}  \alpha_s^{\vee}(u ) \,  \kappa(s) \,s \,	\gamma_{\alpha_s}
	=  \sum_{s\in\mathcal S}  \frac{B(\alpha_s,u)}{\sqrt{ B(\alpha_s,\alpha_s)}} \,  \kappa(s) \rho(\tilde s ) \p.
\end{equation}
By~\eqref{e:comre} and~\eqref{e:combuv}, we have that for $u,v\in V^*$
	\begin{equation}\label{e:Ogamma}
	\scom[+]{ \gamma_u}{\oO_v} = \scom[-]{\beta(u)}{v} - B(u,v) =	\scom[+]{ \gamma_v}{\oO_u}  \p,
\end{equation}
where the last equality follows by Lemma~\ref{l:Buv}.
Expanding the anticommutators in~\eqref{e:Ogamma} gives rise to the following result (cfr.~\cite[Lemma 3.10]{DOV}). 
\begin{lemma}\label{l:Ogammas}
	Let $n \in \{1,2,\dotsc,d\}$, and 
	$u_1,\dotsc,u_n\in V^*$, then
	\begin{equation}\label{e:Ogammas}
		\ca A (\oO_{u_1}	\gamma_{u_2 \dotsm u_n})
		= \ca A (\gamma_{u_1} \oO_{u_2}	\gamma_{u_3 \dotsm u_n})
		= \dotsb 
		= 	\ca A (\gamma_{u_1 \dotsm u_{n-1}}\oO_{u_n})\p.
	\end{equation}
\end{lemma}
\begin{proof}
	The first non-trivial case, for $n=2$, follows immediately from~\eqref{e:Ogamma}: 
	\begin{equation}\label{e:gammaO}
		\oO_u \gamma_v - \oO_v \gamma_u = \gamma_u \oO_v - \gamma_v \oO_u
		\p.
	\end{equation}
	We can then use this to find for general $n\in \{3,\dotsc,d\}$
	\begin{align*}
		\ca A (	\oO_{u_1}	\gamma_{u_2 \dotsm u_n})
		&= 
		\frac{1}{n}\sum_{j=1}^n(-1)^{j-1} \oO_{u_j}  \ca A (\gamma_{u_1 \dotsm\widehat{u}_j \dotsm u_n}) \\
		& = 
		\frac{1}{n(n-1)}\sum_{1\leq j < k \leq n}(-1)^{j+k-1} (\oO_{u_j}\gamma_{u_k}-\oO_{u_k}\gamma_{u_j}) \ca A ( \gamma_{u_1 \dotsm \widehat{u}_j \dotsm \widehat{u}_k\dotsm u_n}) \\
		& = 
		\frac{1}{n(n-1)}\sum_{1\leq j < k \leq n}(-1)^{j+k-1} (\gamma_{u_j}\oO_{u_k}-\gamma_{u_k}\oO_{u_j})  \ca A (\gamma_{u_1 \dotsm \widehat{u}_j \dotsm \widehat{u}_k\dotsm u_n}) \\
		& =  \ca A (\gamma_{u_1} \oO_{u_2}	\gamma_{u_3 \dotsm u_n})\p,
	\end{align*}
	and the other equalities follow by repeated application of the same steps. 
\end{proof}

\subsection{Lie (super)algebra realizations}

The bilinear form $B$ on $V$ naturally corresponds to an element of $(V\otimes V)^*$, a linear map on $V\otimes V$. Since $V$ is finite-dimensional and $B$ is symmetric, we have  $B \in S^2(V^*) \subset V^* \otimes V^*$. As above, let $v_1^*,\dotsc,v_d^*$ denote a basis of $V^*$, dual to a basis $v_1,\dotsc,v_d$ of $V$, then 
\begin{equation}
	B = \sum_{p,q = 1}^d B(v_p,v_q) v_p^* \otimes v_q^* = \sum_{p = 1}^d v_p^* \otimes \beta(v_p) = \sum_{p = 1}^d \beta(v_p) \otimes v_p^* \p,
\end{equation}
which is, by definition, invariant for the action of the group $\rO(V,B)$.

Every element of $\ca{V}$ corresponds to a copy of $V^*$ in $V^*\otimes \ca V$. 
We can use this to map $B\in S^2(V^*)$ in $S^2( V^*\otimes \ca V)$, by viewing $\nu \in \ca V$ as the map $\nu \colon V^* \to V^* \otimes \ca V\colon v \mapsto v \otimes \nu$. 
For $w,z\in\ca{V}$ homogeneous elements for the $\bZ_2$-grading, 
the supersymmetric tensor product is given by 
\begin{equation}
	\label{e:supersymtensor}
	w \odot z = (w \otimes z   + (-1)^{|w||z|} z \otimes w)/2\p,
\end{equation}
and we then consider the following elements of $S^2( V^*\otimes \ca V)$:
\begin{equation}\label{e:Bwz}
	(w \odot z)(B) = \frac12\sum_{p,q=1}^d  B(v_p,v_q)((v_p^* \otimes w)  (v_q^* \otimes z) + (-1)^{|w||z|}  (v_p^* \otimes z )( v_q^* \otimes w ) )	\p.
\end{equation}
Under the quotient by the relations~\eqref{e:comre}, in the superalgebra $A_\kappa$ we have: 
\begin{equation}\begin{aligned}\label{e:Bwz2}
	(w \odot z)(B)  & = 
	\sum_{p,q=1}^d  B(v_p,v_q)((v_p^* \otimes w)(  v_q^* \otimes z)  -  [v_p^* \otimes w , v_q^* \otimes z]/2 )\\
		& =   \sum_{p,q=1}^d  v_p^* \otimes w \,B(v_p,v_q) v_q^* \otimes z -b( w ,  z)\,d/2-\omega( w ,  z) \Omega_\kappa \p,
\end{aligned}\end{equation}
where we used~\eqref{e:BB} and denote
\begin{equation}\label{e:Omega}
	\Omega_\kappa = \sum_{s\in\ca S} \kappa(s) s\p,
\end{equation}
which is a central element in the group algebra $\mathbb{C}[G] $.

In the tensor product of a Weyl and Clifford algebra $\ca W \otimes \ca C$, the space of invariants for the action (as in Definition~\ref{d:actG}) of $\rO$ is generated by the elements of the form $(w \odot z)(B) $~\cite[Proposition~5.11]{CW}. 
\begin{lemma}\label{l:Bg}
	For $w,z\in\ca{V}$, in $A_\kappa$ we have $[(w \odot z)(B) ,\rho(\widetilde G)]=0$.	
\end{lemma}
\begin{proof}
	This follows from Lemma~\ref{l:rhoG} and that $G\subset \rO$ preserves $B$. 
\end{proof}

Next, we consider the adjoint action of elements of the form $(w \odot z)(B) $ on the space $\ca U =V^* \otimes \ca V$ in $A_\kappa$.
Recall that $\gamma \in \ca V_{\bar 1}$ satisfies $b(\gamma,\gamma) =2$. 

\begin{lemma}	\label{l:lemma3}
Let $u\in V^*$, while $\xi_1,\xi_2\in \ca V_{\bar0}$ and	$\eta\in \ca V$. 
In $A_\kappa$,  we have
\begin{align}\label{e:3}	
	[(\xi_1 \odot \xi_2)(B),u\otimes \eta] 
	& = b(\xi_2, \eta ) u \otimes \xi_1    +b(\xi_1, \eta ) u\otimes \xi_2	\p,
	\\ 	
	\label{e:1}
	[({\xi_1 \odot \gamma})(B),u\otimes \eta]  & =  b(\gamma, \eta )u\otimes \xi_1 + b(\xi_1, \eta ) (u\otimes \gamma  +2	\oO_u)
	\p,
\end{align}	
where $	\oO_u$ is given by~\eqref{e:Ov}.
\end{lemma}
\begin{proof}
	Let $\xi,\eta \in  \ca V$ be homogeneous elements for the $\bZ_2$-grading and $\xi_1\in \ca V_{\bar0}$, then by~\eqref{e:Bwz2}
		\[
	[(\xi_1 \odot \xi)(B),u\otimes \eta] =    \sum_{p,q=1}^d B_{pq} [ (v_p^*\otimes \xi_1)( v_q^*\otimes \xi) , u \otimes \eta] - \omega(\xi_1,\xi) \sum_{s\in\ca S} \kappa(s) [s,u\otimes \eta] \p.
	\]	
	First, note that $\omega(\xi_1,\xi) =0$ if $\xi \in \ca V_{\bar1}$, and $	[s,u\otimes \eta]=0$ for $\eta \in \ca V_{\bar1}$. 	
		For  $\eta \in \ca V_{\bar0}$, via~\eqref{e:actO} and~\eqref{e:s} we have
	\[
	[s,u\otimes \eta] = 	((s\cdot u - u  )\otimes \eta )\, s=   - \alpha_s^{\vee}( u ) \,  (\alpha_s\otimes \eta)\, s \p.
	\]
	
	By means of~\eqref{e:comre}, we find
	\begin{align*}
		[ (v^*_p\otimes \xi_1)(  v_q^*\otimes \xi ), u \otimes \eta]  
		= \ & v_p^* \otimes \xi_1 [ v_q^*\otimes \xi , u \otimes \eta]+ (-1)^{|\xi||\eta|} [ v_p^* \otimes \xi_1 , u\otimes \eta ]v_q^*\otimes \xi     \\ 
		= \ &  v_p^* \otimes \xi_1 (B( v_q^*, u )b(\xi, \eta )+\psi_\kappa^B( v_q^*, u )\omega(\xi, \eta )) \\*
		&  + (-1)^{|\xi||\eta|}  (B( v_p^*, u )b(\xi_1, \eta )  +\psi_\kappa^B( v_p^*, u )\omega(\xi_1, \eta ) ) v_q^*\otimes \xi \p.
	\end{align*} 
For $\eta = \gamma$, we have $\omega(\cdot,\gamma)=0$ and the cases where $\eta = \gamma$ now follow by~\eqref{e:BBB}. 

Next, we consider $\eta \in \ca V_{\bar0}$. Let $\xi = \xi_2 \in \ca V_{\bar0}$. We find using~\eqref{e:BB} and~\eqref{e:psiB}
		\begin{align*}
			\sum_{p,q=1}^d  B_{pq}  v_p^*\otimes \xi_1\psi_\kappa^B( v_q^*, u ) 
			&=  \sum_{s\in\mathcal S}\alpha_s^{\vee}( u ) \,  \kappa(s) \alpha_s \otimes \xi_1  s	\p,
			\\	
			\sum_{p,q=1}^d \psi_\kappa^B( v_p^*, u )  B_{pq} v_q^*\otimes \xi_2 	 
			&=  \sum_{s\in\mathcal S} \alpha_s^{\vee}( u ) \,   \kappa(s)\, s \, \alpha_s \otimes \xi_2  
			\p.
		\end{align*} 	
		Now, collecting the appropriate terms, and using that $ \alpha_s$ is a $-1$ eigenvector of $s\in\mathcal S$, we have
		\[
		\sum_{s\in\mathcal S} \alpha_s^{\vee}( u ) \,  \kappa(s) \alpha_s \otimes (\omega(\xi_1, \xi_2 )\eta + \omega(\xi_2, \eta ) \xi_1 - \omega(\xi_1, \eta ) \xi_2  ) s \p,
		\]  
		where $\omega(\xi_1, \xi_2 )\eta  + \omega(\xi_2, \eta ) \xi_1- \omega(\xi_1, \eta ) \xi_2 =0$ for all $\xi_1,\xi_2,\eta\in \ca V_{\bar0} =\bC^2 $.
		
Finally, we consider the case $\xi = \gamma \in \ca V_{\bar1}$ and $\eta = \xi_2 \in \ca V_{\bar0}$. Here, the remaining terms are
	\begin{align*}
[(\xi_1 \odot \gamma)(B),u\otimes \xi_2] &=    \sum_{p,q=1}^d B_{pq}   (B( v_p^*, u )b(\xi_1, \xi_2 )  +\psi_\kappa^B( v_p^*, u )\omega(\xi_1, \xi_2 ) ) v_q^*\otimes \gamma \\
& = b(\xi_1, \xi_2 )u\otimes \gamma  + \omega(\xi_1, \xi_2 ) \sum_{s\in\mathcal S} \alpha_s^{\vee}( u ) \,   \kappa(s)\, s \, \alpha_s \otimes \gamma 
 \p.
	\end{align*} 		
		The desired result now follows by~\eqref{e:Ov}.
	\end{proof}

When restricted to $ \ca V_{\bar0}$, in terms of the basis $x^+,x^-$, the relations of Lemma~\ref{l:lemma3} can be written as follows. 
For $v^-\in V$ and $v^+\in V^*$, in $H_\kappa \subset A_\kappa$, we have 
\begin{align}\label{e:lemma1a}
	[(x^-\odot x^-)(B),v^+] &= 2\beta(v^+)\p,  & [(x^+\odot x^+)(B),v^-]& = -2\beta(v^-)	\p,
	\\\label{e:lemma1b}
	[(x^+\odot x^-)(B),v^+]& = v^+\p,& [(x^+\odot x^-)(B),v^-] &= -v^-	\p.
\end{align}

Rewriting~\eqref{e:1} for $\xi_1=x^-$ and  $\xi_2=x^+$, we get for $u\in V^*$
\begin{equation}\label{e:Ov2}
	\oO_u = \frac12\big([(x^-\odot \gamma)(B),u] - \gamma_u\big) = \frac12\Big(\sum_{p=1}^d [v_{p},u]\gamma_{v^{*}_{p}}  - \gamma_u\Big)\p,
\end{equation}
where $v_1^*,\dotsc,v_d^*$ denotes a basis of $V^*$ dual to a basis $v_1,\dotsc,v_d$ of $V$. Using~\eqref{e:comre}, the expression~\eqref{e:Ov2} reduces to~\eqref{e:Ov}.

We can also consider the map $\oO \colon V^* \to \rho(\widetilde G) \colon u \mapsto \oO_u$, where  $\oO_u$ is given by~\eqref{e:Ov}.
We then have the following result for 
\begin{equation}
(\oO \otimes \gamma)(B)  =\sum_{p,q=1}^d  \oO_{v_p^*} \,B(v_p,v_q) \gamma_{v_q^*} = \sum_{p=1}^d \oO_{x_p} e_p\p.
\end{equation}

\begin{lemma}\label{l:Oug}  In $A_\kappa$, one has
	\[
	(\oO \otimes \gamma)(B) = \Omega_\kappa = 	(\gamma \otimes \oO)(B)\p.
	\]
	
	
\end{lemma}
\begin{proof}
	Using~\eqref{e:Ov} and \eqref{e:BBB}, we have
	\[
		(\oO \otimes \gamma)(B) 
		= \sum_{p,q=1}^d  B(v_p,v_q)\frac12\sum_{s\in\mathcal S}  \alpha_s^{\vee}(v^*_p ) \,  \kappa(s) \,s \,	\gamma_{\alpha_s} \, \gamma_{v_q^*}
			 = \sum_{s\in\mathcal S}   \frac{ \kappa(s) \,s \,	\gamma_{\alpha_s} \, \gamma_{\alpha_s}}{B(\alpha_s,\alpha_s)}
		\p.\qedhere
	\]	
%
\end{proof}

\begin{propo}\label{p:BB}
	For 
	$z_1,z_2,z_3,z_4\in \ca{V}$ homogeneous for the $\bZ_2$-grading, denoting $	w_1  = 	b(z_2,z_3)z_1 + (-1)^{|z_2||z_3|}b(z_1,   z_3)z_2$ and  $w_2 =b( z_2 , z_4)z_1 + (-1)^{|z_2||z_4|}b( z_1 , z_4)z_2$,
	 in $A_\kappa$ one has
	\begin{equation}\label{e:BB1}
		\begin{aligned}~
			& [	(z_1\odot z_2)(B),(z_3\odot z_4)(B)]=   (w_1\odot  z_4)(B) 
			+(-1)^{(|z_1|+|z_2|)|z_3|} (z_3\odot w_2 )(B)  \p.
		\end{aligned}
	\end{equation}
\end{propo}
\begin{proof}
Follows by direct computation using Lemma~\ref{l:lemma3} and the fact that $\dim_{\bC}(\ca{V}_{\bar 1}) = 1$. 
For instance,  using Lemma~\ref{l:lemma3},~\eqref{e:comre}, \eqref{e:BB} and Lemma~\ref{l:Oug}, we have
\begin{align*}
	& \quad  [	(\xi_1\odot \gamma)(B),(\xi_2\odot \gamma)(B)]\\ & =     \sum_{p,q=1}^d B(v_p,v_q) \Big( [	(\xi_1\odot \gamma)(B), v_p^* \otimes \xi_2  ]  v_q^* \otimes \gamma
	+  v_p^* \otimes \xi_2 [	(\xi_1\odot \gamma)(B),  v_q^* \otimes \gamma ] \Big) \\
	 & = \sum_{p,q=1}^d B(v_p,v_q) \Big( b(\xi_1,\xi_2)(v_p^* \otimes\gamma + 2 \oO_{v^*_p}) v_q^* \otimes \gamma 
	 +  (v_p^* \otimes \xi_2) 2( v_q^* \otimes \xi_1)  \Big)\\
	 & = 2 	(\xi_2\odot \xi_1)(B) \p. \qedhere
\end{align*}	
\end{proof}

Proposition~\ref{p:BB} shows that the elements $(w \odot z)(B) $  for $w,z\in \ca V $ form a realization of the Lie superalgebra $\fr{osp}(\ca V,b) \cong \fr{osp}(1|2)$ in $A_\kappa$.
The elements $(w \odot z)(B) $ for $w,z\in \ca V_{\bar{0}}=\bC^2 $  form a realization of the even subalgebra $\fr{sp}(\ca V_{\bar{0}},\omega) \cong \fr{sl}(2)$ in $H_\kappa$.
In particular, using the basis $x^-,x^+,\gamma$ of $\ca V$, we have that the elements
\begin{equation}
	\label{e:osp}
	\begin{aligned}
		F^+ & \colonequals 	\frac{1}{\sqrt2} (x^+ \odot \gamma)(B)
		= 	\frac{1}{\sqrt2}\sum_{p,q=1}^d v^{*}_{p}B_{pq}\gamma_{v^{*}_{q}}\\
			F^- & \colonequals  	\frac{1}{\sqrt2}(x^- \odot \gamma)(B) 
			= 	\frac{1}{\sqrt2}\sum_{p,q=1}^d \beta(v^{*}_{p})B_{pq}\gamma_{v^{*}_{q}}
				= 	\frac{1}{\sqrt2}\sum_{p=1}^d v_{p}\gamma_{v^{*}_{p}}\\
	H &\colonequals 	(x^+ \odot x^-)(B)
	=\sum_{p,q=1}^d v^{*}_{p}B_{pq}\beta(v^{*}_{q}) +\frac{d}{2}+ \Omega_\kappa 
	=\sum_{p=1}^d v^{*}_{p}v_{p} +\frac{d}{2}+ \Omega_\kappa,\\ 
		E^+& \colonequals  \frac12		(x^+ \odot x^+)(B)
		= \frac12\sum_{p,q=1}^d v^{*}_{p}B_{pq}v^{*}_{q},\\ 
		E^-& \colonequals  -\frac12		(x^- \odot x^-)(B)
		= -\frac12\sum_{p,q=1}^d \beta(v^{*}_{p})B_{pq}\beta(v^{*}_{q})
		= -\frac12\sum_{p,q=1}^d v_{p}B_{pq}v_{q}\p,
	\end{aligned}
\end{equation}
	satisfy the commutation relations~\eqref{e:osp12re}.

	\section{Supercentralizers}\label{s:supercent}
	
To describe the supercentralizer of the realization of 	$\mathfrak{osp}(1|2)$ in $A_\kappa$ given by~\eqref{e:osp}, we first look at the centralizer of its even subalgebra $\mathfrak{sl}(2)$.

\subsection{Centralizer of \texorpdfstring{$\mathfrak{sl}(2)$}{sl(2)}}\label{s:sl2}%

In	$H_\kappa\subset A_\kappa$,  we define $M_{uv} = M(u,v) \colonequals  u\beta(v) -  v\beta(u)$ for $u,v\in V^*$. 
Similar to $(w \odot z)(B) $ in~\eqref{e:Bwz}, every element of $V^*$ corresponds to a copy of $\ca V$ in $V^* \otimes \ca V$, or a copy of $\ca V_{\bar0}$ in $V^* \otimes \ca V_{\bar0}$. 
Using now the skew-symmetric form $\omega$, we have for $u,v\in V^*$,
\begin{equation}\label{e:DAMO}
	\begin{aligned}
		(u\wedge v)(\omega) &= \frac12( u\beta(v) -  \beta(u)v - v\beta(u) + \beta(v) u) \\
		&	=  u\beta(v) -  v\beta(u) + \frac12( [\beta(v), u]-[\beta(u),v] ) \\
		&	=  u\beta(v) -  v\beta(u)  \\
		&	=  \beta(v)u -  \beta(u)v \p,
	\end{aligned}
\end{equation}
where we used Lemma~\ref{l:Buv}. 
In the Dunkl operator realization, $M(u,v)$ becomes an angular momentum operator where the partial derivative is replaced by a Dunkl operator, see also~\cite{Feigin,Calvert}.

By means of the relations~(\ref{e:lemma1a}--\ref{e:lemma1b}), 
it is easily verified that the elements of the form $M(u,v)$, for $u,v\in V^*$, commute with $(w \odot z)(B) $, for $w,z\in\ca V_{\bar 0}$. 

In~\cite[Theorem 6.5]{CDM}, the authors proved that the centralizer of the $\fr{sl}(2)$ realization inside $H_\kappa$ is the associative subalgebra of $H_\kappa$  generated by the group $G$ and the Dunkl angular momentum operators, that is
\begin{equation}\label{e:Centsl2}
		\Cent_{H_\kappa}(\fr{sl}(2)) =  \langle M_{uv} \mid u,v \in V^*\rangle  \rtimes G\p,
\end{equation}
where the action of $g\in G$ on $M_{uv}$ is given by $ M_{g\cdot u\,g\cdot v}$ for $u,v \in V^*$ and $g\in G$.
	

The elements $M(u,v)$ for $u,v\in V^*$ generate a deformation of (the associative algebra generated by) the orthogonal Lie algebra $\fr{so}(V,B)\cong\fr{so}(d)$. The proof proceeds in the same way as the one for \cite[Theorem~2.5]{DOV} or \cite[Proposition~6.7]{CDM}.
\begin{propo}\label{p:bbH}
	For 
	$u,v,x,y\in V^*$, in $H_\kappa$ one has
	\begin{equation}\label{e:bbH}
		\begin{aligned}~
			[	M(u,v),M(x,y)]=\ 	&  M(v,x)B_\kappa(u, y)   -M(u,  x)B_\kappa(v, y) \\
			&  -M( v , y)B_\kappa(u, x)   + M( u , y)B_\kappa(v, x) \p,
		\end{aligned}
	\end{equation}	
	where $B_\kappa = B +\psi_\kappa^B$, with the latter given in~\eqref{e:psiB}.
\end{propo}
\begin{proof}Use $	M(u,v) =   u\beta(v) -  v\beta(u)$,  $	M(x,y) =   x\beta(y) -  y\beta(x)$
	\begin{align*}
		& 	[	M(u,v),M(x,y)]= 		
		u [\beta(v),x]\beta(y) - 	v [\beta(u),x]\beta(y)	
		-u [\beta(v),y]\beta(x)	 + 	v [\beta(u),y]\beta(x)	\\
		& \qquad\qquad\qquad\quad-	x [\beta(y),u]\beta(v)	+ 	x [\beta(y),v]\beta(u)	 
		+y [\beta(x),u]\beta(v)	 -	y [\beta(x),v]\beta(u)
	\end{align*}
	which, using~\eqref{e:combuv}, equals
	\begin{align*}
		= 	\ & 	
		M(u,y) [\beta(x),v] - M(v,y) [\beta(u),x]	
		-M(u,x) [\beta(y),v]	 + 	M(v,x) [\beta(u),y]	\\
		&+ 	u ([[\beta(x),v],\beta(y) ]-[[\beta(y),v],\beta(x)])
		- 	v ([[\beta(x),u],\beta(y)]- [[\beta(y),u],\beta(x)])	\\
		&
		-	x ([[\beta(u),y],\beta(v)]	- [[\beta(v),y],\beta(u)])	 
		+y ([[\beta(u),x],\beta(v)]	 -	 [[\beta(v),x],\beta(u)])\p.
	\end{align*}	
	The last two lines vanish by Lemma~\ref{l:uvx}, which we prove next. Hence, the desired result follows	by~\eqref{e:comre}.
\end{proof}

\begin{lemma}\label{l:uvx}
	For $u,v \in V$ and $x^*,y^* \in V^*$, in $H_\kappa$
	\[	
	[[x^*,u],v] = 	[[x^*,v],u]  \p,\qquad [[x^*,v],y^*] = 	[[y^*,v],x^*] 	\p.
	\]	
\end{lemma}
\begin{proof}
	Writing out	$[[x^*,u],v] -	[[x^*,v],u]$ we have
	\[
	(x^*u- ux^*)v-v(x^*u- ux^*) -(x^*v-vx^*)u + u(x^*v-vx^*) \p,
	\]
	where all terms cancel using~\eqref{e:RC}. The other identity follows in the same way.
\end{proof}

	\subsection{Supercentralizer of \texorpdfstring{$\mathfrak{osp}(1|2)$}{osp(1|2)}	}\label{s:osp12}%
	

 
Recall that $P_\pm$ is given by~\eqref{e:Pdelta}, the anti-symmetrization operator by~\eqref{e:asym}, and the quantization map by~\eqref{e:quant}. 
We now define the following elements, which, by Proposition~\ref{p:osp12}, are in the supercentralizer of $\mathfrak{osp}(1|2)$  in $ A_\kappa  = H_\kappa \otimes \ca C$. An	explicit expression is given in Lemma~\ref{l:Oun}.

\begin{defin}\label{d:OA}
	For a positive integer $n$ and $u_1,\dotsc,u_n\in V^*$, we define
	\begin{equation}\label{e:O}
		O_{u_1\dotsm u_n} 
		\colonequals -	P_{\pm} (  q (u_1 \wedge \dotsm \wedge  u_n))/2 = -	P_{\pm} (  \ca A (\gamma_{u_1 \dotsm u_n}))/2 \in A_\kappa \p,
	\end{equation}
which is  skew-symmetric multilinear in its indices.
	
	
	%
\end{defin}

Note that the factor $-1/2$ is chosen to coincide with the definition in~\cite{DOV}, and to have a coefficient of 1 for $M_{uv}$ in~\eqref{e:Ouv2}.

\begin{lemma}\label{l:groupaction}
The group $\rho(\widetilde{G})$ interacts with the elements~\eqref{e:O} as follows
	\begin{equation}
		\rho(\tilde g ) O_{u_1 \dotsb u_n} = (-1)^{|\tilde g|n} O_{p(\tilde g)\cdot u_1 \dotsb p(\tilde g)\cdot u_n} \rho(\tilde g ) \p,
	\end{equation}	
for $\tilde g \in \widetilde G$ and $u_1,\dotsc,u_n\in V^*$.
\end{lemma}
\begin{proof}
	This follows from $P_{\pm} $ being an even element of $U(\fr{osp}(1|2)$ and thus commuting with $\rho(\tilde g)$ by
Lemma~\ref{l:Bg}, that the quantization map~\eqref{e:quant} is a $G$-module isomorphism and Lemma~\ref{l:rhoG}. 
\end{proof}


Next, we want to give an explicit expression for the elements~\eqref{e:O}. Recall that $\oO_{u}$ for $u\in V^*$ is given by~\eqref{e:Ov}. The case $n=1$ of the next result, shows that $O_u = \oO_u$ for $u\in V^*$, see also~\eqref{e:Pgammav}. 

\begin{lemma}\label{l:Pdeltagamma} 
	Let $n \in \{1,2,\dotsc,d\}$, and 
		$u_1,\dotsc,u_n\in V^*$, then
		\begin{align*}
			P_\pm (\gamma_{u_1} \dotsc \gamma_{u_n})  &= 	 (1-n)\gamma_{u_1} \dotsm \gamma_{u_n} -2 \sum_{j=1}^n \gamma_{u_1} \dotsm \oO_{u_j}\dotsm \gamma_{u_n}	
			\\
			& \quad -2\sum_{1\leq j<k \leq n}(-1)^{j+k-1}(u_j\beta(u_k) - \beta(u_j)u_k) \gamma_{u_1 \dotsm \widehat{u}_j \dotsm \widehat{u}_k\dotsm u_n}	\p.
		\end{align*}	
\end{lemma}
\begin{proof}	
		For $v\in V^*$, using~\eqref{e:Pdelta}, \eqref{e:osp}, and~\eqref{e:1}, we have
		\begin{equation}\label{e:Pgammav}
		P_\pm (\gamma_v) =  \gamma_v - [F^-,[F^+,\gamma_v ]]
		= \gamma_v - [(x^-\odot \gamma)(B),v ]
		= -2	\oO_v\p.
		\end{equation}
		Now, for $n\in \{2,\dotsc,d\}$ and 
		$u_1,\dotsc,u_n\in V$,
		we have
		\begin{align*}
			P_\pm (\gamma_{u_1} \dotsm \gamma_{u_n}) = 	\ & \gamma_{u_1} \dotsm \gamma_{u_n} - \left[ F^-,[ F^+,\gamma_{u_1} ]\gamma_{u_2}\dotsm \gamma_{u_n}-\gamma_{u_1}[F^+,\gamma_{u_2}\dotsm \gamma_{u_n} ]\right]\\
			= \ & \gamma_{u_1} \dotsm \gamma_{u_n} - [ F^-,[ F^+,\gamma_{u_1} ]]\gamma_{u_2}\dotsm \gamma_{u_n}- [F^+,\gamma_{u_1}][F^-,\gamma_{u_2}\dotsm \gamma_{u_n} ]
			\\
			& +[F^-,\gamma_{u_1}][F^+,\gamma_{u_2}\dotsm \gamma_{u_n} ]
			-\gamma_{u_1}[F^-,[F^+,\gamma_{u_2}\dotsc \gamma_{u_n} ]]\p.
		\end{align*}
		By definition~\eqref{e:Pdelta} and relation~\eqref{e:1}, this becomes 
		\begin{align*}
			P_\pm (\gamma_{u_1} \dotsm \gamma_{u_n}) =  \ & P_\pm (\gamma_{u_1})\gamma_{u_2}\dotsm \gamma_{u_n}	
			+  \gamma_{u_1}P_\pm (\gamma_{u_2}\dotsm \gamma_{u_n}) 	- \gamma_{u_1} \dotsm \gamma_{u_n} 
			\\
			& 	- 2\sum_{j=2}^n(-1)^{j-2}(u_1\beta(u_j) - \beta(u_1)u_j) \gamma_{u_2}\dotsm  \widehat{\gamma}_{u_j}\dotsm \gamma_{u_n}	\p.
		\end{align*}
		Applying this formula recursively yields the desired result.
	\end{proof}
	
	\begin{lemma}\label{l:Oun} 
	Let $n \in \{1,2,\dotsc,d\}$, and 
	$u_1,\dotsc,u_n\in V^*$, then
	\begin{align*}
&\quad 	O_{u_1\dotsm u_n}  = 	\frac{n-1}{2}  \ca A (\gamma_{u_1\dotsm u_n})
	+ n\,   \ca A (\oO_{u_1}\gamma_{u_2\dotsm u_n})
	+ \frac{n(n-1)}{2}  \ca A ( M_{u_1u_2} \gamma_{u_3\dotsm u_n})
	\\
  & =	-\frac{(n-1)(n-2)}{4} \ca A (\gamma_{u_1\dotsm u_n})
- n(n-2)  \ca A (\oO_{u_1}\gamma_{u_2\dotsm u_n}	)
 + \frac{n(n-1)}{2}  \ca A (O_{u_1u_2} \gamma_{u_3\dotsm u_n})
	\p.
	\end{align*}	
\end{lemma}	

Note that the antisymmetrization in these expressions expands to
\begin{align*}
	O_{u_1\dotsm u_n} 
	&= 	\frac{n-1}{2}  \ca A (\gamma_{u_1\dotsm u_n})
	+ \sum_{j=1}^n (-1)^{j-1} O_{u_j}   \ca A (\gamma_{u_1 \dotsm \widehat{u}_j\dotsm u_n})\nonumber		\\*	
	& \quad +\sum_{1\leq j<k \leq n}(-1)^{j+k-1}M(u_j,u_k)    \ca A (\gamma_{u_1 \dotsm \widehat{u}_j \dotsm \widehat{u}_k\dotsm u_n}) \nonumber
	\\	
	& 	=-\frac{(n-1)(n-2)}{4}  \ca A (\gamma_{u_1\dotsm u_n})
	- (n-2)\sum_{j=1}^n (-1)^{j-1} O_{u_j}   \ca A (\gamma_{u_1 \dotsm \widehat{u}_j\dotsm u_n})		\\*	
	& \quad +\sum_{1\leq j<k \leq n}(-1)^{j+k-1}O_{u_ju_k}  \ca A ( \gamma_{u_1 \dotsm \widehat{u}_j \dotsm \widehat{u}_k\dotsm u_n})
	\p.
\end{align*}	

	\begin{proof}
	The first expression follows by antisymmetrizing the result of Lemma~\ref{l:Pdeltagamma} (multiplied by $-1/2$), using 
	 Lemma~\ref{l:Ogammas} and noting that, by~\eqref{e:combuv},
\[
\ca A (u_j\beta(u_k) - \beta(u_j)u_k)  = (u_j\beta(u_k) - \beta(u_j)u_k - u_k\beta(u_j) + \beta(u_k)u_j)/2 =  M (u_j,u_k) \p.
		\]
		
	For $u,v\in V^*$, the first expression gives
	\begin{equation}\label{e:Ouv2}
		O_{uv} =   \ca A (\gamma_{uv})/2 + 2  \ca A (O_{u}\gamma_{v} )+ M_{uv}  \p.
	\end{equation}
For general $u_1,\dotsc,u_n\in V^*$, we can use~\eqref{e:Ouv2} 
to replace $M(u_j,u_k)$ in the lemma's first expression for $O_{u_1\dotsm u_n} $ to find the second.
\end{proof}

\begin{lemma}\label{l:O2gammas} 
	Let $n \in \{2,\dotsc,d\}$, and 
	$u_1,\dotsc,u_n\in V^*$, then
	\begin{equation}\label{e:O2gammas}
	\ca A (O_{u_1u_2}	\gamma_{u_3 \dotsm u_n})
		= \ca A (\gamma_{u_1} O_{u_2u_3}	\gamma_{u_4 \dotsm u_n})
		= \dotsb 
		= 	\ca A (\gamma_{u_1 \dotsm u_{n-2}}O_{u_{n-1}u_n})\p.
	\end{equation}
\end{lemma}
\begin{proof} 
	This follows by means of~\eqref{e:Ouv2} and then using  Lemma~\ref{l:Pdeltagamma} and the fact that 
	$M(u_j,u_k)$ commutes with all factors of $  \ca A (\gamma_{u_1 \dotsm \widehat{u}_j \dotsm \widehat{u}_k\dotsm u_n}) $. 
\end{proof}
\begin{lemma}\label{l:Oun2} 
	Let $n \in \{3,\dotsc,d\}$, and 
	$u_1,\dotsc,u_n\in V^*$, then		
	\[	
		(n-3)	O_{u_1\dotsm u_n} 	= - 4 (n-2)\ca A ( O_{u_1} O_{u_2 \dotsm u_n}) + 2(n-1)  \ca A (O_{u_1u_2} O_{u_3\dotsm u_n})
			\p.
	\]
\end{lemma}
Note that expanding the antisymmetrization gives the following expressions
\begin{align*}
	\frac{n(n-3)}{4}	O_{u_1\dotsm u_n} &= - (n-2)\sum_{j=1}^n (-1)^{j-1} O_{u_j} O_{u_1 \dotsm \widehat{u}_j\dotsm u_n}		 \\ &\quad +\sum_{1\leq j<k \leq n}(-1)^{j+k-1}O_{u_ju_k}  O_{u_1 \dotsm \widehat{u}_j \dotsm \widehat{u}_k\dotsm u_n}\p.
\end{align*}
\begin{proof} 	
	The result follows by applying $-P_\pm/2 $ to the second expression of Lemma~\ref{l:Oun} and using Lemma~\ref{l:Pdelta}.
	\end{proof}
	
The previous result shows that all elements of the form~\eqref{e:O} can be written in terms of those having one, two or three indices.

In particular, for $u,v\in V^*$, by~\eqref{e:Auv}, \eqref{e:Ouv2} becomes
	\begin{equation}\label{e:Ouv}
		\begin{aligned}[]
			O_{uv} &=  u\beta(v) - \beta(u)v + (\gamma_u\gamma_v+B(u,v))/2 + \oO_u\gamma_v + \gamma_u \oO_v\\
			&=   u\beta(v) - v\beta(u) + (\gamma_u\gamma_v-B(u,v))/2 + \oO_u\gamma_v - \oO_v\gamma_u
	\p.
		\end{aligned}
	\end{equation}
Note also that 
	\begin{equation}
	\label{e:POuv}
	- P_\delta (\gamma_u\gamma_v) /2= O_{uv} -  B(u,v)/2\p.
\end{equation}
	For $u,v,w\in V^*$,  by~\eqref{e:Auvw} we have 
	\begin{equation}\label{e:Ouvw}
		\begin{aligned}[]
			O_{uvw} =\ &    \ca A (\gamma_{uvw})+  M(v,w)\gamma_u  -M(u,w)\gamma_v +  M(u,v)\gamma_w \\ & 
			 +  \oO_u  \ca A (\gamma_{vw})  - \oO_v  \ca A (\gamma_{uw})   + \oO_w  \ca A (\gamma_{uv}) 
			\p.
		\end{aligned}
	\end{equation}
	
		
		\begin{propo}\label{p:ospcent}
			For $\mathfrak{osp}(1|2)$ realized in $A_\kappa = H_\kappa \otimes \ca C$ by the elements~\eqref{e:osp},
			its supercentralizer $\Cent_{A_\kappa}(\mathfrak{osp}(1|2))$ is generated by $\rho(\widetilde G)$  and the elements $O_{uv}$ and $O_{uvw}$ for $u,v,w\in V^*$.
		\end{propo}
		\begin{proof}
			
			By Proposition~\ref{p:osp12}, we obtain the centralizer of $\mathfrak{osp}(1|2)$ by applying the operator $P_\pm$ to $\Cent_{A_\kappa}(\mathfrak{osp}(1|2)_{\bar 0})$. 			
			In $H_\kappa$, the centralizer $\Cent_{H_\kappa}(\mathfrak{osp}(1|2)_{\bar 0})$ of the even subalgebra $\mathfrak{osp}(1|2)_{\bar 0} \cong \fr{sl}(2)$ is  generated by $\{M_{uv}\mid u,v \in V^*\}$ and the group $G$, see~\ref{e:Centsl2}.
			As $H_\kappa$ does not interact with the Clifford algebra part of $A_\kappa$, we have 
			\begin{equation}\label{e:Centsl2c}
				\Cent_{A_\kappa}(\mathfrak{osp}(1|2)_{\bar 0}) =  \Cent_{H_\kappa}(\mathfrak{osp}(1|2)_{\bar 0})  \otimes  \ca C\p.
			\end{equation}
			
			First, we note that	this means the elements $O_{u_1\dotsm u_n}$ for $u_1,\dotsc,u_n\in V^*$ are in $\Cent_{A_\kappa}(\mathfrak{osp}(1|2))$ by Definition~\ref{d:OA}.
Also, $\rho(\bC[\widetilde G]) \subset 	\Cent_{A_\kappa}(\mathfrak{osp}(1|2))$ by Lemma~\ref{l:Bg}.

			Now, a general element of $\Cent_{A_\kappa}(\mathfrak{osp}(1|2)_{\bar 0})$ can be written as a sum of terms of the form
			$
			m\,c\, g
			$
			where $m$ is a product of elements of $\{M_{uv}\mid u,v \in V^*\}$, $c \in \ca C$ and $g\in G$. 		
	Assume $m$ is a non-zero product of at least one element of $\{M_{uv} \mid u,v \in V^*\}$.  Then, there are $u,v\in V^*$ such that $m = M_{uv} m'$ with $m'$  a product of elements of $\{M_{uv} \mid u,v \in V^*\}$, one fewer  than $m$.
			
			By~\eqref{e:Ouv} and using Lemma~\ref{l:Pdelta}, we can write 
			\begin{align*}
				P_\pm(m\,c\, g) =\ &	P_\pm(M_{uv}m'\,c\, g)\\
				=\ &	P_\pm((O_{uv} -(\gamma_{uv}-B(u,v))/2 - \oO_u\gamma_v + \oO_v\gamma_u )m'\,c\, g) \\
				=\ &	(O_{uv}+B(u,v)/2)P_\pm(m'\,c\, g) -P_\pm(m'\,\gamma_{uv} c\, g)/2\\ 
				&- \oO_uP_\pm(m'\,\gamma_v c\, g)+ \oO_vP_\pm( m'\,\gamma_uc\, g)\p,
			\end{align*}
		since $O_{uv},\oO_u,\oO_v$ are in $	\Cent_{A_\kappa}(\mathfrak{osp}(1|2))$.  
			Hence, by proving the property for the case where $m$ is a constant, the result follows by induction. 
			
			Note that
			for the case $m\,c\, g = M(u,v)$ we find in this way
			\begin{equation}\label{e:Pomegauv}
					P_\pm(M_{uv} )
					=
					2 O_{uv}+2 \oO_u\oO_v-2 \oO_v\oO_u\p.
			\end{equation}
			
	An element $g\in G$ can be written as a product of reflections in $\ca S$. For each $s\in \ca S$, since $\gamma_{\alpha_s}^2 = B(\alpha_s,\alpha_s) \neq 0$ in $\ca C$,  we can write 
	\[
	s =\gamma_{\alpha_s}^2 s /B(\alpha_s,\alpha_s) = \gamma_{\alpha_s}\rho(\tilde s) /\sqrt{B(\alpha_s,\alpha_s)}\p.
	\]
	In this way, we can write $g = c_g\, \tilde g$ where $c_g\in \ca C$ and $\tilde g\in \tilde G$. 	
	By Lemma~\ref{l:Pdelta}, we have 
	\[
	P_\pm( c\,g) = P_\pm( c\,c_g \,\tilde g) = P_\pm( c\,c_g )\,\tilde g\p.
	\]
	In particular, for $s\in\ca S$ this becomes
	\begin{equation}\label{e:Ps}
		P_\pm(s) = -2 \,\oO_{\alpha_s} \rho(\tilde s) /\sqrt{B(\alpha_s,\alpha_s)}  \p.
	\end{equation}		
			
The elements of $P_\pm(\Cent_{A_\kappa}(\mathfrak{osp}(1|2)_{\bar 0}))$ are thus given by products of elements in $\rho(\widetilde G)$ and elements of the form $O_{u_1\dotsm u_n}$ for $u_1,\dotsc,u_n\in V^*$ 
	Lemma~\ref{l:Oun2} 	then shows that the latter can be generated by means of those having one, two or three indices, which completes the proof.			
		\end{proof}
		
Given the interaction with 
the group $\rho(\widetilde{G})$ in Lemma~\ref{l:groupaction}, a minimal set of generators for $\Cent_{A_\kappa}(\mathfrak{osp}(1|2))$ requires only a subset of the elements  $O_{uv}$ and $O_{uvw}$ for $u,v,w\in V^*$.
		
		
		\subsection{Supercentralizer algebra relations}\label{s:rels}
		
		We can apply $P_\pm $ to both sides of an equality to determine relations for elements of the centralizer $O_{\kappa}$ in this realization. 
		We recall that throughout the $\bZ_2$-graded commutator is used $[a,b] = ab - (-1)^{|a||b|} ba$, where  $a,b$ are homogeneous elements for the $\bZ_2$-grading, and we also denote 
			\begin{equation}\label{e:ascom}
		\ascom{a}{b} = ab + (-1)^{|a||b|} ba\p.
	\end{equation}

	The following result includes a generalization of	\cite[Theorem 3.13]{DOV}.
		\begin{propo}\label{p:OujOun}
		Let $n \in \{2,\dotsc,d\}$, and 
		$u_1,\dotsc,u_n\in V^*$, then
		\begin{equation*}\label{e:Ogammas2}
	\ca A(	O_{u_1}O_{u_2\dotsm  u_{n}}	)  =	\ca A(	O_{u_1\dotsm  u_{n-1}}O_{u_{n}}	 ) 	
	\quad\text{ or }\quad
		\ca A(	\scom{O_{u_1}}{O_{u_2\dotsm  u_{n}}	} ) 
		= 0\p,
	\end{equation*}	
and
	\begin{equation*}\label{e:O2gammas2}
	\ca A(O_{u_1u_2}	O_{u_3 \dotsm u_n})
= 		\ca A(O_{u_1 \dotsm u_{n-2}}O_{u_{n-1}u_n})
	\quad\text{ or }\quad
	\ca A(	\scom{O_{u_1u_2}}{	O_{u_3 \dotsm u_n}})
	=0\p.
\end{equation*}
	\end{propo}
	\begin{proof} 
		By 	Lemma~\ref{l:Ogammas}, we have $	\ca A(	O_{u_1}	\gamma_{u_2\dotsm  u_n})= 		\ca A(\gamma_{u_1\dotsm  u_{n-1}}	O_{u_n})$. 
	 The first result follows by applying $-P_\pm/2 $ and using Lemma~\ref{l:Pdelta}. To see that this implies	$\ca A(	[O_{u_1},O_{u_2\dotsm  u_{n}}	] ) = 0$, we expand the antisymmetrization 	$\ca A(	O_{u_1}O_{u_2\dotsm  u_{n}}	)  =	\ca A(	O_{u_1\dotsm  u_{n-1}}O_{u_{n}}	 )$ 	 as
		\begin{equation*}
	\frac{1}{n}	\sum_{j=1}^n (-1)^{j-1}	O_{u_j}	O_{u_1\dotsm \widehat{u}_j \dotsm u_n}
		= 	\frac{1}{n}\sum_{j=1}^n (-1)^{n-j}	O_{u_1\dotsm \widehat{u}_j \dotsm u_n}	O_{u_j}\p.
	\end{equation*}	 

	Similarly, by Lemma~\ref{l:O2gammas} we have  
	$	\ca A(O_{u_1u_2}	\gamma_{u_3 \dotsm u_n})
	= 		\ca A(\gamma_{u_1 \dotsm u_{n-2}}O_{u_{n-1}u_n})$
and again the result follows by applying $-P_\pm/2 $.
	\end{proof}
Specific cases of the previous proposition are, 	for $u,v,w,z\in V^*$, 
	\begin{align}
&	\scom[-]{ O_{uv} }{ O_{w} } -  \scom[-]{ O_{uw} }{ O_v } + \scom[-]{  O_{vw} }{ O_u }= 0\\  
& \scom[+]{O_{uvw} }{ O_z }  - \scom[+]{O_{uvz} }{ O_w } + \scom[+]{O_{uwz} }{ O_v } - \scom[+]{O_{vwz} }{ O_u } =  0\p, \label{e:OuvwOx}
\end{align}
Note also that by Proposition~\ref{p:OujOun}, one can rewrite antisymmetrized products using~\eqref{e:ascom}, for instance $\ca A (O_{u_1u_2} O_{u_3 u_4}) = \ca A (\ascom[+]{O_{u_1u_2}}{ O_{u_3 u_4}})/2$.

\begin{lemma}\label{l:O45}
For 
$u_1,\dotsc,u_5\in V^*$,
	\begin{align*}	
	O_{u_1\dotsm u_4}  
	 & = 6 \, \ca A (O_{u_1u_2} O_{u_3 u_4})- 8\, \ca A ( O_{u_1u_2 u_3 } O_{u_4})  
\end{align*}	
and
\begin{align*}	
	O_{u_1\dotsm u_5} 	&= 4  \ca A (O_{u_1u_2u_3} O_{u_4 u_5})
	+48 \ca A ( O_{u_1u_2 u_3 } O_{u_4} O_{u_5}) 
	- 36\ca A ( O_{u_1u_2} O_{u_3 u_4}O_{u_5})
\p.
\end{align*}			
\end{lemma}
\begin{proof}
	By Lemma~\ref{l:Oun2} and Proposition~\ref{p:OujOun}, we have
	\[	
(n-3)	O_{u_1\dotsm u_n} 	=2(n-1)  \ca A (O_{u_1u_2} O_{u_3\dotsm u_n}) - 4 (n-2)\ca A ( O_{u_1 \dotsm u_{n-1}} O_{u_n}) 
\p.
\]
For $n=4$, this gives the desired result. For $n=5$, we have
		\[	
	2	O_{u_1\dotsm u_5} 	= 8  \ca A (O_{u_1u_2} O_{u_3u_4 u_5})- 12\ca A (  O_{u_1 \dotsm u_{4}} O_{u_5}) 
	\p,
	\]	
	where we can use the $n=4$ result to obtain the desired result.
\end{proof}

\begin{propo}
	For $a,b,u,v\in V^*$, we have
				\begin{align*}
			\scom[-]{O_{ab}}{O_{uv}}	=\ & B(b,u)( O_{av} + \GG{a}{v} )
		- B(a,u)(O_{bv}+\GG{b}{v})\\ &
		- B(b,v)( O_{au} + \GG{a}{u} )
		+ B(a,v)(O_{bu}+\GG{b}{u})\\ &
		+(\scom[+]{O_a}{O_{buv}}
		- \scom[+]{O_b}{O_{auv}}   +  \scom[+]{O_{abu}}{O_v}  -\scom[+]{O_{abv}}{O_u}   )/2
		\p,
	\end{align*}
	 or, denoting 
		$\hat u = 	B(b,u)a - B(a,u)b$ and 	$\hat v = 	B(b,v)a - B(a,v)b$, 
	\begin{align*}
		\scom[-]{O_{ab}}{O_{uv}}	 &= O_{\hat u v} + \ascom[-]{O_{\hat u}}{O_v} + \scom[+]{O_{abu}}{O_v} \\ & \quad+ O_{u\hat v} + \ascom[-]{O_u}{O_{\hat v}} 				
		- \scom[+]{O_{abv}}{O_u}\p.
	\end{align*}	
\end{propo}
\begin{proof}
	Using the definition~\eqref{e:Ouv} together with~\eqref{e:gammaO} and~\eqref{e:adso}
	\begin{align*}
	[O_{ab},\ca A(\gamma_{uv})] 
	= \ &   [ \gamma_{ab}/2 , \gamma_{uv}] 
				+[ O_a\gamma_b - O_b\gamma_a,\gamma_{uv} ] \\
	=\ &   [\gamma_a\gamma_b, \gamma_u]\gamma_v/2 
				+
				\gamma_u[\gamma_a \gamma_b, \gamma_v]/2 \\ &
				+ (O_a\gamma_b
				- O_b\gamma_a)\gamma_u\gamma_v 
				- \gamma_u \gamma_v ( \gamma_a O_b  
				-  \gamma_b O_a) \\ 
	=\ &  (B(b,u) \gamma_a
				- B(a,u)\gamma_b)\gamma_v
				+
				\gamma_u(B(b,v)\gamma_a 
				-
				B(a,v) \gamma_b) \\ &
				+ O_a\gamma_b\gamma_u \gamma_v 
				- O_b\gamma_a\gamma_u\gamma_v 
				- \gamma_u \gamma_v  \gamma_a O_b  
				+ \gamma_u \gamma_v  \gamma_b O_a \p.
			\end{align*}		
			By~\eqref{e:O} and \eqref{e:POuv}, 
		 applying $-P_\pm/2 $ to the previous computation gives
			\begin{align*}
				[O_{ab},O_{uv}] 
				=\ &  B(b,u) (O_{av} - B(a,v)/2 ) - B(a,u)(O_{bv} -B(b,v)/2)\\
				&
				+
				B(b,v)(O_{ua} -B(u,a)/2)
				-
				B(a,v) (O_{ub} -B(u,b)/2)\\ &
				+ O_a(O_{buv} +B(u,v)O_b-B(b,v)O_u+B(b,u)O_v)\\ &
				- O_b(O_{auv} +B(u,v)O_a-B(a,v)O_u+B(a,u)O_v)\\ &
				-(O_{uva} +B(v,a)O_u-B(u,a)O_v+B(u,v)O_a) O_b  \\ &
				+ (O_{uvb} +B(v,b)O_u-B(u,b)O_v+B(u,v)O_b)O_a \p.
			\end{align*}
		Collecting the appropriate terms and using~\eqref{e:OuvwOx} gives the desired results.
\end{proof}


The next two propositions contain expressions with four or five indices, which are given in Lemma~\ref{l:O45}. 
\begin{propo}\label{p:O2O34}
	Let $a,b,c,u,v,w\in V^*$, and denote $\hat x = B(b,x) a - B(a,x)b$ for $x \in \{c,u,v,w\}$. We have  
	\begin{align*}
		\scom[-]{ O_{ab} }{ O_{uvw} } &=  O_{\hat u vw} + O_{u\hat  vw}  + O_{ uv\hat w}   \\*
		& \quad + 
		\ascom[+]{ O_{\hat u}}{ O_{vw}} - 
		\ascom[+]{ O_{\hat v}}{ O_{uw}} + 
		\ascom[+]{ O_{\hat w}}{ O_{uv}}\\* 	& \quad + \scom[-]{O_a}{O_{buvw} } - \scom[-]{O_b}{ O_{auvw} }\p,
	\end{align*}
and 
	\begin{align*}
		\scom[-]{O_{ab}}{O_{cuvw}} 	&=   
		O_{\hat cuvw} + O_{c\hat uvw} + O_{cu\hat vw} + O_{cuv\hat w}\\*	
& \quad + 	\ascom[-]{O_{\hat c}}{O_{uvw}}   + 	\ascom[-]{O_{\hat u}}{O_{cvw}} + 	\ascom[-]{O_{\hat v}}{O_{cuw}} + 	\ascom[-]{O_{\hat w}}{O_{cuv}}
			\\* & \quad
		\scom[+]{O_a}{O_{bcuvw} }	 -\scom[+]{O_b}{O_{acuvw}}	
		\p.
	\end{align*}
\end{propo}
\begin{proof}
For the first relation, using~\eqref{e:Ouv} and~\eqref{e:gammaO}
	\[
	[O_{ab},\ca A(\gamma_{xyz}) ] =	
		[\gamma_{uv}/2,\ca A(\gamma_{xyz}) ]  + 	[O_u\gamma_v - O_v\gamma_u,\ca A(\gamma_{xyz}) ]  \p.
	\]
On the one hand, by~\eqref{e:adso} we have
\begin{align*}	
	 \scom{\gamma_{uv}/2}{\ca A(\gamma_{xyz}) } & = \gamma_{\hat x yz} + \gamma_{x\hat  yz}  + \gamma_{ x y\hat z}  \\
	 & \quad - B(y,z)\gamma_{\hat x}+ B(x,z)\gamma_{\hat y}- B(x,y)\gamma_{\hat z} \\
	 & = \ca A(\gamma_{\hat x yz} )+ \ca A(\gamma_{x\hat  yz})  + \ca A(\gamma_{ x y\hat z} ) 
\end{align*}	
where we used that 
\[
B(\hat x , z ) = B(v,x) B(u , z ) - B(u,x)B(v, z )  = - B(x,\hat z)\p.
\]
On the other hand, 
	\begin{align*}
& \quad \ \scom{O_u\gamma_v - O_v\gamma_u}{\ca A(\gamma_{xyz}) }  \\
& = (O_u\gamma_v - O_v\gamma_u)\ca A(\gamma_{xyz}) - \ca A(\gamma_{xyz}) (\gamma_uO_v - \gamma_vO_u)\\
& =\scom{O_u}{\ca A(\gamma_{vxyz}) } - \scom{O_v}{\ca A(\gamma_{uxyz}) } \\
& \quad + 
\ascom{ O_u }{ B(v,x)\ca A(\gamma_{yz}) - B(v,y)\ca A(\gamma_{xz}) + B(v,z)\ca A(\gamma_{xy}) }\\
& \quad - 
\ascom{ O_v}{ B(u,x)\ca A(\gamma_{yz}) - B(u,y)\ca A(\gamma_{xz})+ B(u,z)\ca A(\gamma_{xy}) }\\
& = 
\scom{O_u}{\ca A(\gamma_{vxyz}) } - \scom{O_v}{\ca A(\gamma_{uxyz}) } \\
& \quad + 
\ascom{ B(v,x)O_u - B(u,x)O_v}{ \ca A(\gamma_{yz})} \\
& \quad - 
\ascom{ B(v,y)O_u - B(u,y)O_v}{ \ca A(\gamma_{xz})} \\
& \quad + 
\ascom{ B(v,z)O_u - B(u,z)O_v}{ \ca A(\gamma_{xy})}\p.
	\end{align*}		
The result follows after applying  $-P_\pm/2 $.

	Similarly, for the second relation, we have 
\[
\scom[-]{O_{ab}}{\ca A ( \g{cuvw})} 	=   
[\cA{ab}/2 + O_a \g{b} - O_b\g{a} , \cA{cuvw} ]
\]
where 
\[
[\cA{ab}/2 , \cA{cuvw} ]
=	 \cA{\hat cuvw} + \cA{c\hat uvw} +\cA{cu\hat vw} +\cA{cuv\hat w} 
\p,
\]
and
\begin{align*}
	&\quad\	(O_b \g{c} - O_c\g{b})\cA{auvw} - \cA{auvw} (\g{b}O_c-\g{c} O_b  )\\
	& = \scom[+]{O_b}{ \ca A ( \g{cauvw})}
	- B(c,w) \ascom[-]{O_b}{\ca A ( \g{auv})} 	\\ & \quad
	-\scom[+]{O_c}{\ca A ( \g{bauvw} )}	-B(b,v) \ascom[-]{O_c}{\ca A ( \g{auw})}   		\p,	
\end{align*}
while
\[
\scom[-]{O_{ab}}{\ca A ( \g{cuvw})} 	=   
[\cA{ab}/2 + O_a \g{b} - O_b\g{a} , \cA{cuvw} ]
\]
where 
\[
[\cA{ab}/2  , \cA{cuvw} ]
=B(a,u)\cA{bcvw}	  +	B(b,c) \cA{auvw} +  B(b,v)\cA{acuw}	 
\p,
\]
and
\begin{align*}
	&\quad\	(O_a \g{b} - O_b\g{a})\cA{cuvw}
	- \cA{cuvw} (\gO{a}{b})\\
	& = \scom[+]{O_a}{\ca A ( \g{bcuvw} )}	-B(b,c) \ascom[-]{O_a}{\cA{uvw}} -B(b,v) \ascom[-]{O_a}{\ca A ( \g{cuw})} 
	\\ & \quad
	-\scom[+]{O_b}{ \ca A ( \g{acuvw})}
	- B(a,u) \ascom[-]{O_b}{\ca A ( \g{cvw})} 
	\p.			
\end{align*}
The results follow after applying $-P_\pm/2 $.
\end{proof}

Moreover, for brevity assuming also $B(a,c) = 0 = B(b,c)$, we have

\begin{propo}\label{p:O3O3}
	Let $a,b,c,u,v,w\in V^*$ be such that the only $B$-pairings between $\{a,b,c\}$ and $\{u,v,w\}$ that can be non-zero are $B(a,u),B(b,v),B(c,w)$. We have 
		\begin{align*}
		\scom[+]{O_{abc}}{O_{uvw}} 	&=   
		B(a,u) (O_{bcvw} 	+  \ascom[+]{O_{bc}}{ O_{vw}  }	 )
		+	\scom[+]{O_a}{O_{bcuvw} } 
			\\	& \quad
			+ B(b,v) (O_{acuw} + \ascom[+]{O_{ac}}{  O_{uw} } ) 
			-  \scom[+]{ O_b }{O_{acuvw} } 
			\\ & \quad	
			+ B(c,w) (O_{abuv} 	 +\ascom[+]{O_{ab} }{  O_{uv}  }  )
			+  \scom[+]{ O_c }{O_{abuvw} }
				\\ & \quad
		 +  B(b,v)B(c,w) \scom[+]{O_a}{O_u}
		 +   B(a,u)B(c,w)\scom[+]{O_b}{O_v}
		  \\
		 & \quad 
		 +B(a,u)B(b,v)\scom[+]{O_c}{  O_w}
		 	 -  B(a,u) B(b,v)B(c,w)/2
	\end{align*}
where expressions for elements of the form $O_{bcvw}$ and $O_{bcuvw}$ are given in Lemma~\ref{l:O45}. 
\end{propo}
\begin{proof}
We have
	\begin{align*}
		\scom[+]{O_{abc} }{  \ca A ( \gamma_{uvw} ) } & =
	\scom[+]{ -\ca A ( \gamma_{abc} )/2 - 3\ca A (O_{a}\gamma_{bc}) + 3\ca A (O_{ab}\gamma_c ) }{ \ca A ( \gamma_{uvw} ) } \p.
		\end{align*}		
	Now, using~\eqref{e:Auvw} and~\eqref{e:Auvwx}
	\begin{align*}
	\scom[+]{ -\ca A ( \gamma_{abc} )/2  }{   \ca A ( \gamma_{uvw} ) }
& =
-B(a,u) \ca A ( \gamma_{bcvw} ) - B(b,v) \ca A ( \gamma_{acuw} ) - B(c,w) \ca A ( \gamma_{abuv} ) \\ & \quad  + B(a,u) B(b,v)B(c,w) 
 \p,
\end{align*}
while 		$\scom[+]{ - 3\ca A (O_{a}\gamma_{bc})  }{ \ca A ( \gamma_{uvw} ) } $ becomes 
	\begin{align*}	
 & 
 -\scom[+]{O_a}{ \ca A ( \gamma_{bcuvw} )} 
 - \ascom[-]{O_a}{ B(b,v) \ca A ( \gamma_{cuw} ) + B(c,w) \ca A ( \gamma_{buv} )}
  +   \scom[+]{O_a}{ B(b,v)B(c,w)\gamma_u}
  \\ & 
+  \scom[+]{ O_b }{\ca A ( \gamma_{acuvw} )} 
- \ascom[-]{O_b}{ B(a,u) \ca A ( \gamma_{cvw} ) + B(c,w) \ca A ( \gamma_{auv} ) }
	+  \scom[+]{O_b}{ B(a,u)B(c,w) \gamma_v}
  \\ & 
-  \scom[+]{ O_c }{\ca A ( \gamma_{abuvw} )} 
+ \ascom[-]{O_c}{ B(a,u) \ca A ( \gamma_{bvw} ) + B(b,v) \ca A ( \gamma_{auw} )} 
+ \scom[+]{O_c}{ B(a,u)B(b,v) \gamma_w}
\end{align*}	
and	
	\begin{align*}
		\scom[+]{ 3\ca A (O_{ab}\gamma_{c})  }{ \ca A ( \gamma_{uvw} ) } 
	& =
	\scom[-]{O_{ab}}{ \ca A ( \gamma_{cuvw} ) }+ B(c,w)\ascom[+]{O_{ab} }{  \ca A ( \gamma_{uv} ) }
	\\ & \quad
	-   \scom[-]{O_{ac} }{\ca A ( \gamma_{buvw} ) } + B(b,v) \ascom[+]{O_{ac}}{  \ca A ( \gamma_{uw} )}
	\\ & \quad
	+  \scom[-]{ O_{bc}}{ \ca A ( \gamma_{auvw} )} +B(a,u)  \ascom[+]{O_{bc}}{ \ca A ( \gamma_{vw} ) }
	\p.
\end{align*}	
The result follows after applying $-P_\pm/2 $ and using Proposition~\ref{p:O2O34}.
\end{proof}

\subsubsection{Special bases}

We will use the following notational conventions when a vector of the chosen bases of Section~\ref{s:bases} appears as a subscript index in an element of the form~\eqref{e:O}: 

An element of the set  $\{1,\dotsc,d\}$ will be used to refer to the corresponding element of $\{x_p\}_{p=1}^d$. For instance, $O_{12} = O_{x_1x_2}$.

For $\ell=\lfloor  d/2\rfloor$, an element of the set $\{1,\dotsc,\ell\}$ with a $+$ or $-$ above it, will be used refer to the corresponding element of $\{z_p^\pm\}_{p=1}^\ell$, which is (part of) the $B$-isotropic basis of $V^*$.
If $d$ is odd, an index $0$  will be used to refer to $z_0$.
For instance, $O_{\ib{+}{1}\ib{-}{1}0} = O_{z_1^+z_1^-z_0^{~}}  $

The relations determined above, reduce to the following for elements of the basis $\{x_p\}_{p=1}^d$. The proofs are also easier in this case, so they are included for completeness.  
\begin{corol}
	For $i,j,k,l,m,n$ distinct elements of the set $\{1,\dotsc,d\}$,	
\begin{align}
	\scom[-]{O_{ij}}{O_{ki}}  
	&= O_{jk}+  \GG{i}{j}
	+[O_{ijk},O_i]
	\\ 
	[O_{ij},O_{kl}]  
	&= \frac12( [O_i,O_{jkl}]
	-[O_j,O_{ikl}]-[O_{ijl},O_k]
	+[O_{ijk},O_l]
	)\\
	&= 	[O_i,O_{jkl}]
	-[O_j,O_{ikl}]
	\p,
\end{align}	
where the last equality follows by~\eqref{e:OuvwOx}.
	\begin{align}
	[O_{jk},O_{lmn} ]  & =  [O_j, O_{klmn}] - [O_k,O_{jlmn}]\\
	[O_{jk},O_{jlm} ]  & = -O_{klm} - \ascom{O_k}{O_{lm}} - [O_j,O_{jklm}]\\
	[O_{jk},O_{jkl} ]  & =  -\ascom{O_j}{O_{jl}} - \ascom{O_k}{O_{kl}} \p.
\end{align}	
	\begin{align}\label{e:24}
\scom[+]{O_{ijk}}{O_{ijk}} & =   2\left((O_{i})^2+ 	(O_{j})^2  + 	
								(O_{k})^2 + (O_{ij})^2 + (O_{ik})^2 + 		(O_{jk})^2\right) -\frac12\\ 
	\scom[+]{O_{ijk}}{O_{ijl}} & =  [O_k,O_l] 
	    	  + \ascom{O_{ik}}{O_{il}}+ \ascom{O_{jk}}{O_{jl}}\\
\scom[+]{O_{ijk}}{O_{imn}} & =  O_{jkmn}  	+ \ascom{O_{jk}}{O_{mn}}
	   	  +\scom{O_i}{O_{ijkmn}}\\
	\scom[+]{O_{ijk}}{O_{lmn}} 
 			& = 
			\scom{O_i}{ O_{jklmn}} 
			- \scom{O_j }{ O_{iklmn}}
			+\scom{O_k }{ O_{ijlmn}} \label{e:27}	
			\p.
	\end{align}	
\begin{align}
\scom{O_{jk}}{ O_{jklm}} 
& =  - \ascom{O_j}{ O_{jlm}}  - \ascom{O_k}{O_{klm}} \\
\scom{O_{jk}}{ O_{jlmn}} 
& = - O_{klmn}- \ascom{O_k}{O_{lmn}}   - \scom{O_j}{ O_{jklmn}}  \\
\scom{O_{ij}}{ O_{klmn}} 	& =  \scom{O_i}{ O_{jklmn}}  -   \scom{O_j}{ O_{iklmn}} \p.
\end{align}	
\end{corol}
\begin{proof}
For a 3-element set $\{a,b,c\} \subset \{1,\dotsc,d\}$, Proposition~\ref{p:OA2} below gives
	\begin{equation}\label{e:Ouvw2}
		(O_{abc})^2 = -\frac14 + (O_{a})^2+ (O_{b})^2  + (O_{c})^2 + (O_{ab})^2 + (O_{ac})^2 + (O_{bc})^2\p. 
	\end{equation}	
For the second relation, we compute
	\begin{align*}
			[O_{jkl},e_{jkm}] 
			  & = [ - e_{jkl}/2 - 3\ca A (O_{j}e_{kl}) + 3\ca A (O_{jk}e_{l} ),e_{jkm}] \\
			& = 
			-(O_j e_{kl} - O_k e_{jl}+O_l e_{jk} )e_{jkm} - e_{jkm}(e_{jk} O_l -e_{jl} O_k + e_{kl} O_j)	\\*	&\quad
			+(O_{jk} e_{l} -O_{jl} e_{k} + O_{kl} e_j)e_{jkm} + e_{jkm}(e_{j} O_{kl} - e_{k} O_{jl}+e_{l} O_{jk} )
			\\
			& =  O_je_{jlm} -e_{jlm}O_j + O_ke_{klm} -e_{klm}O_k +O_le_m +e_mO_l \\*
			& \quad + O_{jk} e_{jklm} - e_{jklm}O_{jk}  + O_{jl} e_{jm} + e_{jm}O_{jl}+ O_{kl} e_{km} + e_{km}O_{kl}\p.
		\end{align*}
The result follows after applying $-P_\pm/2 $ to both sides and using Proposition~\ref{p:O2O34}.

Similarly, the final two relations follow from
\begin{align*}
& \quad 	[O_{jkl},e_{jmn}] 
	 = [ - e_{jkl}/2 - 3\ca A (O_{j}e_{kl}) + 3\ca A (O_{jk}e_{l} ) ,e_{jmn}] \\
	& =   - e_{klmn}  
	-(O_j e_{kl} - O_k e_{jl}+O_l e_{jk} )e_{jmn} 
	- e_{jmn}(e_{jk} O_l -e_{jl} O_k + e_{kl} O_j)	\\*	&\quad
	+(O_{jk} e_{l} -O_{jl} e_{k} + O_{kl} e_j)e_{jmn} 
	+ e_{jmn}(e_{j} O_{kl} - e_{k} O_{jl}+e_{l} O_{jk} )
	\\
	& =   - e_{klmn} 
	-O_je_{jklmn} -e_{jklmn}O_j
	- O_ke_{lmn} + e_{lmn}O_k 
	+O_le_{kmn} -e_{kmn}O_l  \\*
	& \quad 
	- O_{jk} e_{jlmn} + e_{jlmn}O_{jk}   
	+ O_{jl} e_{jkmn} - e_{jkmn}O_{jl}
	+ O_{kl} e_{mn} + e_{mn}O_{kl}
\p,
\end{align*}				
and
	\begin{align*}
		[O_{ijk},e_{lmn}] 
		& = [ - e_{ijk}/2 - 3\ca A (O_{i}e_{jk}) + 3\ca A (O_{ij}e_{k} ),e_{lmn}] \\
		& = 
		-(O_i e_{jk} - O_j e_{ik}+O_k e_{ij} )e_{lmn} - e_{lmn}(e_{ij} O_k -e_{ik} O_j + e_{jk} O_i)	\\*	&\quad
		+(O_{ij} e_k -O_{ik} e_j + O_{jk} e_i)e_{lmn}
		+ e_{lmn}(e_i O_{jk} - e_j O_{ik}+e_k O_{ij} )\p,
	\end{align*}
after applying $-P_\pm/2 $ and using Proposition~\ref{p:O2O34}.
	\end{proof}

A slight modification of the previous proof also yields	the following relations	
\begin{align*}
	\ascom[-]{O_{jkl}}{O_{jkm}}  & = -O_{lm} - \ascom{O_l}{O_m} 
	+ \ascom{O_{jk} }{ O_{jklm}}\\
	\ascom[-]{O_{jkl}}{O_{jmn}}  & =  - \ascom{O_j}{O_{jklmn}}  
	+ \ascom{O_{jk} }{ O_{jlmn}} + \ascom{O_{jl}}{O_{jkmn}}	
	\p.
\end{align*}				
Together with the relations~(\ref{e:24}--\ref{e:27}) this shows that in $A_\kappa$ a product of elements of the form $O_{uvw}$ for $u,v,w\in V^*$ can be reduced to terms containing at most a single 3-index element.

\begin{propo}\label{p:OA2}
	For	$n \in \{1,\dotsc,d\}$ and $A = \{ a_{1},a_{2},\dotsc,a_{n}\}\subset \{1,\dotsc,d\}$
	\begin{align*}
		(O_A)^2 
		= (-1)^{n(n-1)/2}\bigg( \frac{(n-1)(n-2)}{8} - (n-2)\sum_{a\in A} (O_{a})^2 -\sum_{\{a,b\}\subset A } ({O}_{ab})^2  \bigg)\p.
	\end{align*}
\end{propo}
\begin{proof}
	Note that $(e_A)^2=(-1)^{n(n-1)/2}$ if $|A|=n$. 
	By the last expression of Lemma~\ref{l:Oun} 
	\begin{align*}
		O_A e_A & = -\frac{(n-1)(n-2)}{4}  (e_A)^2
		- (n-2)\sum_{a\in A} O_{a} e_{a}(e_A)^2 
		-\sum_{\{a,b\}\subset A } {O}_{ab} 
		e_ae_b(e_A)^2	\p.
		\p.
	\end{align*}
	Applying  $-P_\pm/2 $ to both sides and using Lemma~\ref{l:Pdelta} yields the desired result.
\end{proof}

\begin{propo}
	For	$u,v,w\in V^*$, we have
	\begin{equation}\label{e:OD}
		\ascom[(-)^d]{ O_{1\dotsm d} }{O_u} = 0\p,\qquad 
		\scom[-]{ O_{1\dotsm d} }{O_{uv}} = 0\p,\qquad 
		\ascom[(-)^d]{ O_{1\dotsm d} }{O_{uvw}} = 0\p.
	\end{equation}		
	
	Moreover, the expression 
	\begin{equation}\label{e:central}
		\Omega = 	(d-2)\sum_{j=1}^d (O_{j})^2  +\sum_{1\leq j<k \leq d } ({O}_{jk})^2  
	\end{equation}		
	is central in $\Cent(\fr{osp}(1|2))$. 

\end{propo}	
Note that for $d$ odd, the relations~\eqref{e:OD} imply that $O_{1\dotsm d}$ commutes (but not supercommutes as it has odd $\bZ_2$-degree) with everything in $\Cent(\fr{osp}(1|2))$.
\begin{proof}
	 The element $e_{1\dotsm d}$ antisupercommutes with the generators of $\ca C$.
	In this way, for	$u,v,w\in V^*$, using the expressions~\eqref{e:Ov}, \eqref{e:Ouv}, \eqref{e:Ouvw}, we find
	\[
	\ascom[(-)^d]{ e_{1\dotsm d} }{O_u} = 0\p,\qquad 
	\scom[-]{ e_{1\dotsm d} }{O_{uv}} = 0\p,\qquad 
	\ascom[(-)^d]{ e_{1\dotsm d} }{O_{uvw}} = 0\p,
	\]	
	so the relations~\eqref{e:OD} follow after applying  $-P_\pm/2 $. 
	
	By Proposition~\ref{p:OA2}, we have
	\begin{align*}
		(O_{1\dotsm d})^2 =  	(-1)^{d(d-1)/2}\bigg( \frac{(d-1)(d-2)}{8} - (d-2)\sum_{j=1}^d (O_{j})^2  -\sum_{1\leq j<k \leq d } ({O}_{jk})^2  \bigg)\p.
	\end{align*}	
	Using again the properties of $e_{1\dotsm d}$, the element $O_{1\dotsm d}$ equals
	\[
	O_{1\dotsm d} = -\frac12P_\delta(e_{1\dotsm d}) = -\frac12(e_{1\dotsm d} - [F^-,[F^+,e_{1\dotsm d}]]) = (F^-F^+ - F^+F^- - 1/2)e_{1\dotsm d}
	\]	
	so, up to the sign $(-1)^{d(d-1)/2}$, the expression $O_{1\dotsm d} e_{1\dotsm d}$ equals the $\fr{osp}(1|2)$
	Scasimir element \eqref{e:SCasiosp}. 
	Since $-P_\pm( O_{1\dotsm d} e_{1\dotsm d})/2 = O_{1\dotsm d}^2 $, we find by~\eqref{e:SCasiosprel} that, up to constants, $\Omega$ equals the $\fr{osp}(1|2)$
	Casimir element \eqref{e:Casiosp} in $ U(\fr{osp}(1|2))$ and thus is central in $\Cent(\fr{osp}(1|2))$.	
\end{proof}	

Using~\eqref{e:Ouvw2}, we have
\[
\sum_{i<j<k}^d (O_{ijk})^2 = -\frac{d(d-1)(d-2)}{24}  + \frac{(d-1)(d-2)}{2}\sum_{j=1}^d (O_{j})^2  + (d-2)\sum_{ j<k  }^d ({O}_{jk})^2  \p,
\]	
so the following combination is also central
\[
\frac{(d-1)(d-2)}{2}\sum_{j=1}^d (O_{j})^2   +\sum_{1\leq j<k \leq d } ({O}_{jk})^2  + \sum_{i<j<k}^d (O_{ijk})^2  \p.
\]

\begin{lemma}\label{l:OO}  In $A_\kappa$, one has
	\[
	(\oO \otimes \gamma)(B) = \Omega_\kappa = 	(\gamma \otimes \oO)(B)\p.
	\]
	
	
\end{lemma}
\begin{proof}
	Using~\eqref{e:Ov} and \eqref{e:BBB}, we have
	\begin{align*}
		(\oO \otimes \gamma)(B) & =\sum_{p,q=1}^d  \oO_{v_p^*} \,B(v_p,v_q) \gamma_{v_q^*} \\	
		& = \sum_{p,q=1}^d  B(v_p,v_q)\frac12\sum_{s\in\mathcal S}  \alpha_s^{\vee}(v^*_p ) \,  \kappa(s) \,s \,	\gamma_{\alpha_s} \, \gamma_{v_q^*}\\
		& = \sum_{s\in\mathcal S}   \,  \kappa(s) \,s \,	\gamma_{\alpha_s} \, \gamma_{\alpha_s}/B(\alpha_s,\alpha_s)
		\p.\qedhere
	\end{align*}	
	%
\end{proof}

\begin{corol}
	For $i,j,k,r,s,t$ distinct elements of the set $\{1,\dotsc,d\}$,	
	\begin{align*}
		[O_{ij},O_{ki}]  
		&= O_{jk}+  \GG{i}{j}
		+[O_{ijk},O_i]
		\\ 
		[O_{ij},O_{kl}]  
		&= \frac12( [O_i,O_{jkl}]
		-[O_j,O_{ikl}]-[O_{ijl},O_k]
		+[O_{ijk},O_l]
		)\\
		&= 	[O_i,O_{jkl}]
		-[O_j,O_{ikl}]
		\p,
	\end{align*}	
	where the last equality follows by~\eqref{e:OuvwOx}.
	
	\begin{align*}
		[O_{jk},O_{lmn} ]  & =  [O_j, O_{klmn}] - [O_k,O_{jlmn}]\\
		[O_{jk},O_{jlm} ]  & = -O_{klm} - \ascom{O_k}{O_{lm}} - [O_j,O_{jklm}]\\
		[O_{jk},O_{jkl} ]  & =  -\ascom{O_j}{O_{jl}} - \ascom{O_k}{O_{kl}} \p.
	\end{align*}
	
	\begin{align*}
		\scom[+]{O_{ijk}}{O_{ijk}} & =   2\left((O_{i})^2+ 	(O_{j})^2  + 	
		(O_{k})^2 + (O_{ij})^2 + (O_{ik})^2 + 		(O_{jk})^2\right) -\frac12\\ 
		\scom[+]{O_{ijk}}{O_{ijt}} & =  [O_c,O_t] 
		+ \ascom{O_{ik}}{O_{it}}+ \ascom{O_{jc}}{O_{jt}}\\
		\scom[+]{O_{ijk}}{O_{ist}} & =  O_{jkst}  	+ \ascom{O_{jk}}{O_{st}}
		+\scom{O_i}{O_{ijkst}}\\
		\scom[+]{O_{ijk}}{O_{rst}} 
		& = 
		\scom{O_i}{ O_{jkrst}} 
		- \scom{O_j }{ O_{ikrst}}
		+\scom{O_j }{ O_{ijrst}} 	
		\p.
	\end{align*}	
	\begin{align*}
		\scom{O_{jk}}{ O_{jklm}} 
		& =  - \ascom{O_j}{ O_{jlm}}  - \ascom{O_k}{O_{klm}} \\
		\scom{O_{jk}}{ O_{jlmn}} 
		& = - O_{klmn}- \ascom{O_k}{O_{lmn}}   - \scom{O_j}{ O_{jklmn}}  \\
		\scom{O_{ij}}{ O_{klmn}} 	& =  \scom{O_i}{ O_{jklmn}}  -   \scom{O_j}{ O_{iklmn}} \p.
	\end{align*}	
\end{corol}

For the $B$-isotropic basis we have the following. Fix $j,k \in \{1\dotsc,\ell\}$ with $j\neq k$ and let $u\in X$ such that $B(z^\pm_j,u) = 0 = B(z^\pm_k,u)$, then
\begin{align*}
	[O_{\ib{+}{j}z_j^-},O_{z_j^\pm u}]  
	&= \pm  2 (O_{z_j^\pm u}+  \GG{z_j^\pm}{u})
	-[O_{z_j^+z_j^-u},O_{z_j^\pm}]
	\\ 
	[O_{z_j^\pm z_k^\pm},O_{z_j^\mp u}]  
	&= \pm  2 (O_{z_j^\pm u}+  \GG{z_j^\pm}{u})
	-[O_{z_j^+z_j^-u},O_{z_j^\pm}]
	\\ 
	[O_{ij},O_{kl}]  
	&= \frac12( [O_i,O_{jkl}]
	-[O_j,O_{ikl}]-[O_{ijl},O_k]
	+[O_{ijk},O_l]
	)\\
	&= 	[O_i,O_{jkl}]
	-[O_j,O_{ikl}]
	\p,
\end{align*}

\subsection{Generalized symmetries} \label{s:gensyms2}

Recall \eqref{e:Q}, given in Section~\ref{s:gensyms}, and the $\fr{osp}(1|2)$-elements~\eqref{e:osp}.
Denote $X \colonequals (x^+ \odot \gamma)(B) = \sum_{j=1}^d x_je_j$ and  $D  \colonequals (x^- \odot \gamma)(B) = \sum_{j=1}^d y_je_j$. 

\begin{propo}\label{p:gensym2}
	For $u \in V^*$, the element
	\begin{equation}\label{e:gensym}
		R_{u} \colonequals Q^-(\gamma_u) = (H-1)\gamma_u + 
	  X \, \beta(u)
	\end{equation}
is a generalized symmetry of $D$.
\end{propo}
\begin{proof}
	Follows by Proposition~\ref{p:gensym}, using~\eqref{e:1}	of Lemma~\ref{l:lemma3} to work out
	\begin{equation}
 Q^{-}(\gamma_u)=(H-1)\gamma_u+ F^{+}  [ F^{-}, \gamma_u]	\p.	
	\end{equation}
	\end{proof}

\subsection{Representations and Howe correspondence}
\marginpar{TODO: put in TeX}

\section*{Appendix}

\end{document}